\documentclass[12pt, a4paper]{amsart}

\usepackage{amssymb, amsmath}
\usepackage[all]{xy}
\usepackage[inline]{enumitem}
\usepackage{hyphenat}
\usepackage[colorlinks,linkcolor=blue,citecolor=blue,urlcolor=teal]{hyperref}
\usepackage{xcolor}
\usepackage{mathtools}
\usepackage{tikz-cd}
\usepackage[labelformat=empty]{caption}
 \usepackage[ left=3cm, right=3cm, top = 2.8cm, bottom = 2.8cm ]{geometry}

\makeatletter
\def\@tocline#1#2#3#4#5#6#7{\relax
  \ifnum #1>\c@tocdepth 
  \else
    \par \addpenalty\@secpenalty\addvspace{#2}%
    \begingroup \hyphenpenalty\@M
    \@ifempty{#4}{%
      \@tempdima\csname r@tocindent\number#1\endcsname\relax
    }{%
      \@tempdima#4\relax
    }%
    \parindent\z@ \leftskip#3\relax \advance\leftskip\@tempdima\relax
    \rightskip\@pnumwidth plus4em \parfillskip-\@pnumwidth
    #5\leavevmode\hskip-\@tempdima
      \ifcase #1
      \or\or \hskip 2em \or \hskip 2em \else \hskip 3em \fi%
      #6\nobreak\relax
    \dotfill\hbox to\@pnumwidth{\@tocpagenum{#7}}\par
    \nobreak
    \endgroup
  \fi}
\makeatother

\newcommand{\eq}[2]{\begin{equation}\label{#1}#2 \end{equation}}
\newcommand{\ml}[2]{\begin{multline}\label{#1}#2 \end{multline}}
\newcommand{\mlnl}[1]{\begin{multline*}#1 \end{multline*}}

\newcommand{\arir}{\ar@{^{(}->}}
\newcommand{\aril}{\ar@{_{(}->}}
\newcommand{\are}{\ar@{>>}}

\newcommand{\xr}[1] {\xrightarrow{#1}}

\newcommand{\xrightarrowdbl}[2][]{ \xrightarrow[#1]{#2}\mathrel{\mkern-14mu}\rightarrow}

\newtheorem{lem}{Lemma}[section]
\newtheorem{thm}[lem]{Theorem}

\newtheorem{prop}[lem]{Proposition}

\newtheorem{cor}[lem]{Corollary}

\theoremstyle{definition}
\newtheorem{defn}[lem]{Definition}
\newtheorem{defn-prop}[lem]{Definition-Proposition}

\newtheorem{nota}[lem]{Notation}
\newtheorem{para}[lem]{}

\newtheorem*{acknowledgement}{Acknowledgement}

\theoremstyle{remark}

\newtheorem{rmk}[lem]{Remark}

\newtheorem{exs-rmks}[lem]{Examples and Remarks}

\newtheorem{claim}{Claim}[lem]
\newtheorem*{claim*}{Claim}

\newcounter{zaehler} 
\numberwithin{equation}{lem}

\newcommand{\Q}{\mathbb{Q}}
\newcommand{\Z}{\mathbb{Z}}

\renewcommand{\P}{\mathbf{P}}
\newcommand{\A}{\mathbf{A}}

\newcommand{\G}{\mathbf{G}}

\newcommand{\sF}{\mathcal{F}}
\newcommand{\sG}{\mathcal{G}}

\newcommand{\sO}{\mathcal{O}}

\newcommand{\sU}{\mathcal{U}}

\newcommand{\fp}{\mathfrak{p}}

\newcommand{\fr}{\mathfrak{r}}

\newcommand{\Xb}{{\overline{X}}}

\newcommand{\Db}{{\overline{D}}}
\newcommand{\Vb}{{\overline{V}}}

\newcommand{\tF}{{\widetilde{F}}}

\newcommand{\tY}{{\widetilde{Y}}}

\newcommand{\ux}{{\underline{x}}}
\newcommand{\uy}{{\underline{y}}}
\newcommand{\uz}{{\underline{z}}}

\newcommand{\RSC}{{\operatorname{\mathbf{RSC}}}}
\newcommand{\RSCNis}{{\operatorname{\mathbf{RSC}}}_{\Nis}}

\newcommand{\ul}[1]{{\underline{#1}}}

\newcommand{\Hom}{\operatorname{Hom}}

\newcommand{\Ker}{\operatorname{Ker}}

\renewcommand{\Im}{\operatorname{Im}}

\newcommand{\Tr}{\operatorname{Tr}}
\newcommand{\Nm}{\operatorname{Nm}}
\newcommand{\Div}{\operatorname{Div}}

\newcommand{\Spec}{\operatorname{Spec}}

\newcommand{\dlog}{\operatorname{dlog}}

\newcommand{\Sm}{\operatorname{\mathbf{Sm}}}

\newcommand{\red}{{\operatorname{red}}}

\newcommand{\Zar}{{\operatorname{Zar}}}
\newcommand{\Nis}{{\operatorname{Nis}}}

\newcommand{\et}{{\operatorname{\acute{e}t}}}

\newcommand{\inj}{\hookrightarrow}

\newcommand{\surj}{\rightarrow\!\!\!\!\!\rightarrow}

\newcommand{\Res}{\operatorname{Res}}

\newcommand{\id}{{\operatorname{id}}}

\newcommand{\CH}{{\operatorname{CH}}}
\newcommand{\Frac}{{\operatorname{Frac}}}

\newcommand{\e}{{\epsilon}}
\renewcommand{\b}{{\rm b}}
\renewcommand{\c}{{\rm c}}
\newcommand{\mc}{{\rm mc}}

\newcommand{\gen}{{\rm gen}}

\newcommand{\ol}{\overline}

\renewcommand{\epsilon}{\varepsilon}

\newcommand{\la}{\langle}
\newcommand{\ra}{\rangle}

\newcommand{\ulMCor}{\operatorname{\mathbf{\underline{M}Cor}}}

\def\Fh{F^h}
\def\sOh{\sO^h}
\def\tx{\tilde{x}}

\title[Higher local symbols of reciprocity sheaves]{Ramification theory of reciprocity sheaves, II\\Higher local symbols}
\author{Kay R\"ulling \and Shuji Saito}
 
\address{Bergische Universit\"at Wuppertal\\ Gau\ss str. 20, D-42119 Wuppertal, Germany}
\email{ruelling@uni-wuppertal.de}
 
\address{Graduate School of Mathematical Sciences, University of Tokyo, 3-8-1 Komaba, Tokyo 153-8941, Japan}
\email{sshuji@msb.biglobe.ne.jp}

\thanks{K.R.\ was supported by the DFG Heisenberg Grant RU 1412/2-2. 
S.S.\ is supported by the JSPS KAKENHI Grant (20H01791). }
\begin{document}

\begin{abstract}
We construct  a theory of higher local symbols along Par\v{s}in chains for reciprocity sheaves.  
Applying this formalism to differential forms, gives a new construction 
of the Par\v{s}in-Lomadze residue maps, and applying it to the torsion characters of the fundamental group gives back
the reciprocity map from Kato's higher local class field theory in the geometric case.  
The higher local symbols satisfy various reciprocity laws.
The main result of the  paper is a characterization of the modulus attached to a section of a reciprocity sheaf in terms of the 
higher local symbols. 
\end{abstract}

\maketitle

\tableofcontents

\def\Keta{K_\eta}

\section{Introduction}

In this note we apply the results from \cite{RS-ZNP} to obtain a theory of higher local symbols for 
reciprocity sheaves. These symbols are  higher dimensional generalizations of the local symbols defined by Rosenlicht-Serre \cite{Serre-GACC}
in the 1-dimensional case for commutative  algebraic groups. Higher local symbols are defined along Par\v{s}in chains and satisfy various reciprocity laws. Applying this formalism to differential forms, gives a new construction 
of the Par\v{s}in-Lomadze residue maps, and applying it to the torsion characters of the fundamental group gives back
the reciprocity map from Kato's higher local class field theory in the geometric case.  
The main result of the  paper is a characterization of the modulus attached to a section of a reciprocity sheaf in terms of the 
higher local symbols. This result will be an essential ingredient in  \cite{RS-AS} and \cite{SaitologRSC}.

\medskip
\begin{para}
We fix a perfect field $k$.
Reciprocity sheaves were introduced by Kahn, Saito, and Yamazaki in \cite{KSY2}. A reciprocity sheaf $F$ is a 
Nisnevich sheaf with transfers which has the additional property, that any section $a\in F(U)$ over a smooth 
$k$-scheme $U$ has a modulus, i.e., there is a proper $k$-scheme $X$ and 
an effective Cartier divisor $D$ on $X$, such that $U=X\setminus D$ and the pair $(X,D)$ measures the defect of $a$ 
being regular outside of  $U$. Though one should think of the modulus 
as a measure for the pole order or the depth of ramification of $a$ along $D$,
 the interest comes from the fact that it is defined in a motivic way,
namely by requiring  that the action of certain finite correspondences is zero on $a$, see \ref{para:RSC} 
and the references there for details. 
The subgroup of sections of $F(U)$ with modulus $(X,D)$ is denoted by $\tF(X,D)$. 
If $X$ is projective of dimension $d$ over $k$, then in  \cite[Proposition 6.7]{RS-ZNP} we  construct a pairing
\eq{intro:0}{(-,-)_{(X,D)_K/K}:\tF(X_K,D_K)\otimes_{\Z} H^d(X_{K,\Nis}, K^M_d(\sO_{X_K}, I_{D_K}))\to F(K),}
where $K$ is a function field over $k$, $X_K=X\otimes_k K$,  $K^M_d$ is Kerz' improved Milnor $K$-theory sheaf, and 
$K^M_d(\sO_{X_K}, I_{D_K})=\Ker(K^M_{d, X_K}\to K^M_{d, D_K})$.
For particular reciprocity sheaves  this gives back several pairings  which were constructed in the literature by
different methods, e.g., if $F=\Hom_{\rm cont}(\pi^{\rm ab}(-), \Q/\Z)$ and $K$ is a finite field, this pairing 
(or at least the pro-system over larger and larger $D$) was constructed in \cite{KS-GCFT}  
to obtain higher dimensional geometric class field theory, 
or if $k$ has characteristic zero and $F(U)$ denotes the absolute rank  one connections on $U$,
this pairing was constructed by Bloch-Esnault in the case $U$ has dimension 1, see \cite[(4.8)]{Bloch-Esnault}. 
A disadvantage of the motivic definition of $\tF(X,D)$ is that it is hard to decide which sections  of $F(U)$ have modulus 
$(X,D)$.
To study the pairing in other interesting examples, e.g., $F(U)=H^1(U_{\rm fppf}, G)$ for $G$ a finite $k$-group scheme,
or $F(U)=H^0(U, R^{n+1}\e_*\Q/\Z(n))$ with $\Q/\Z(n)$ the \'etale motivic complex of weight $n$ with $\Q/\Z$-coefficient and $\e: X_{\et}\to X_{\Nis}$ the change of sites map, it is desirable to get a better hold on $\tF(X,D)$. Easier-to-handle global descriptions of
$\tF(X,D)$ are given in  \cite{RS-ZNP} and \cite{RS-AS}.
In the present article we give a purely local description,
at least under certain extra assumptions on $(X,D)$. 
\end{para}

\begin{para}
Let $K$ be a function field over $k$.
Recall that a Par\v{s}in chain (or maximal chain) on an integral  finite-type $K$-scheme $X$ of dimension $d$
is a sequence $\ux=(x_0, \ldots, x_d)$ of points of $X$ with $x_i<x_{i+1}$, i.e., $x_i$ is a strict specialization of $x_{i+1}$, for all $i=0,\ldots, d-1$.
Let $F$ be a reciprocity sheaf.  For any maximal chain $\ux$ on $X$, we define in section \ref{sec:HLS} the higher local symbol
\eq{intro:1}{(-,-)_{X/K,\ux}: F(K^h_{X,\ux})\otimes_{\Z} K^M_d(K_{X,\ux}^h)\to F(K),}
where $K^h_{X,\ux}$ is the henselization of $\sO_{X,x_0}$ along the chain $\ux$, see \ref{para:hlr} for details. 
The definition of this pairing relies on the map
$c_{\ux}: K^M_d(K_{X,\ux}^h)\to H^d(X, j_!K^M_{d,U})$ already considered in \cite{KS-GCFT} and 
the pushforward $H^d(X_\Nis, j_!F\la d\ra_U)\to F(K)$ constructed in \cite{RS-ZNP} (and relying on the pushforward
constructed in \cite{BRS}). Using the natural map $K(X)\inj K^h_{X,\ux}$
\eqref{intro:1} also induces a semi-local pairing
\[(-,-)_{X/K,\ux}: F(K(X))\otimes_{\Z} K^M_d(K(X))\to F(K).\]
The family of these symbols (for all $X$ and all $\ux$) is uniquely determined by the properties
\ref{HS1}-\ref{HS4} which resemble the properties used by  Serre to characterize and construct 
his local symbols on curves for commutative algebraic groups in \cite[III]{Serre-GACC}. 
This uniqueness property can be used to check that the higher local symbols coincide for $F=\Omega^q$ with those defined by
Par\v{s}in and Lomadze (\cite{Parsin}, \cite{Lo}), for  details on this and further examples see \ref{exs-rmks:res}. 
The property \ref{HS3} roughly says that the symbol $(-,-)_{X/K,\ux}$ vanishes on 
\[\tF(X_K,D_K)\otimes K_d^M(\sO_{X_K}, I_{D_K})_{(x_0,\ldots, x_{d-1})}^h,\]
where $K_d^M(\sO_{X_K}, I_{D_K})_{(x_0,\ldots, x_{d-1})}^h$ is the Nisnevich stalk of $K_d^M(\sO_{X_K}, I_{D_K})$ at the chain $(x_0,\ldots, x_{d-1})$,
see \eqref{para:hls.Nisnevichstalk}.
The property \ref{HS4} is the reciprocity theorem
\[\sum_{x_{i-1}<y<x_{i+1}} (a, \beta)_{X/K,(x_0,\ldots, x_{i-1},y,x_{i+1}, \ldots, x_d)}=0, \]
for all  $a\in F(K(X))$, $\beta\in K^M_d(K(X))$, and $i\in \{0,\ldots, d\}$, where in the case $i=0$, we have to assume $X$ projective.
Interestingly,  in \cite{Lo} (and many similar constructions) property \ref{HS3} follows easily from the definition of the local symbol and 
the reciprocity law \ref{HS4} is a theorem, whose proof requires a more involved argument, 
whereas in our case \ref{HS4} is a formal consequence  of the construction 
and \ref{HS3} follows from the pairing \eqref{intro:0}, which is one of the main results from \cite{RS-ZNP}.
\end{para}

The  main result of the present paper is the following theorem (the statement in the body of the text is a bit stronger).

\begin{thm}[see Theorem \ref{thm:LS}, Proposition \ref{prop:diag}]\label{intro:thm}
Let $X$ be a smooth $k$-scheme of pure dimension $d$ and $D$ an effective Cartier divisor on $X$ whose support has simple normal crossings.
Let $U=X\setminus|D|$ and $a\in F(U)$.
Assume that there exists an open dense immersion $X\inj \Xb$ into a smooth and projective $k$-scheme, such that 
$(\Xb\setminus U)_{\rm red}$ has simple normal crossings. 
Let $V\subset X$  be an open neighborhood of the generic points of $|D|$.
Then the following conditions are equivalent:
\begin{enumerate}[label=(\roman*)]
\item\label{thm:intro1} $a\in \tF(X,D)$.
\item\label{thm:intro3} For any function field $K/k$ and any maximal chain $\ux=(x_0,x_1\ldots, x_d)$ on $V_K$ 
with  $x_{d-1}\in D_K$, we have
\[(a_K,\beta)_{X_K/K,\ux}=0,\quad \text{for all } \beta\in K^M_d(\sO_{X_K}, I_{D_K})_{x_{d-1}},\]
where $X_K=X\otimes_k K$ and $a_K\in F(U_K)$ denotes the  pullback of $a$.
\end{enumerate}
If furthermore $D$ is a {\em reduced} divisor with simple normal crossings, then the same is true without assuming the
existence of the smooth projective compactification $\Xb$ with SNCD boundary.
\end{thm}

If $F$ has level $\le 3$ (see \ref{para:level}) one can also get around the assumption on the existence of the smooth compactification with SNCD boundary,
see Corollary \ref{cor:HLS}.

The proof of Theorem \ref{intro:thm} uses the main results from \cite{RS-ZNP}. 
The stronger statement for  $D$ reduced relies on \cite[Corollary 2.5]{SaitologRSC}, \cite{S-purity},
 and an additional diagonal argument explained in section \ref{sec:diag}.
Theorem \ref{intro:thm} and the properties \ref{HS1}-\ref{HS5} of the higher local symbols  play an important role in the proofs of the main result of
\cite{RS-AS} and  in the proof of \cite[Theorem 4.2]{SaitologRSC}.

\begin{acknowledgement}
The authors thank the referee for his comments which helped to clarify the exposition.
\end{acknowledgement}

\begin{nota}\label{nota}
\begin{enumerate}
\item In this paper $k$ denotes a field and $\Sm$ the category of separated schemes 
which are smooth and of finite type over $k$. For $k$-schemes $X$ and $Y$ we write $X\times Y:= X\times_k Y$.
For $n\ge 0$ we write $\P^n:=\P^n_k$, $\A^n:=\A^n_k$.
\item Let $F$ be a Nisnevich sheaf on a scheme $X$ and $x\in X$ a point.
Then we denote by $F_x$ its Zariski stalk and by $F_x^h=\varinjlim_{x\in U/X}F(U)$ the Nisnevich stalk, where the limit is over 
all Nisnevich neighborhoods $U\to X$ of $x$.

\item For a reduced ring $R$, ${\rm Frac}(R)$ denotes its total ring of fractions. 

\item For a scheme $X$ we denote by $X_{(i)}$ (resp. $X^{(i)}$)  the set of $i$-dimensional (resp. $i$-codimensional) points of $X$.

\end{enumerate}
\end{nota}

\section{Preliminaries on reciprocity sheaves and pairings}
This paper builds on \cite{RS-ZNP}. In this section we recall  some notations and results, see
{\em loc. cit.} and the references there for more details.

In this section $k$ is a perfect field.
\begin{para}\label{para:RSC}
A {\em modulus pair} $(X,D)$ in the sense of \cite{KMSY1}, \cite{KMSY2} consists of a separated scheme of finite type over $k$ and 
an effective (possibly empty) Cartier divisor $D$ on $X$, such that the complement $U=X\setminus D$ is smooth over $k$. 
The modulus pair $(X,D)$ is called {\em proper}, if $X$ is proper over $k$.
Let $F$ be a presheaf with transfers and $U\in\Sm$. A {\em modulus for an element $a\in F(U)$} is  a proper modulus pair $(X,D)$ with $U=X\setminus D$,
such that  for all $S\in \Sm$ and all integral closed subschemes $Z\subset \A^1\times S\times U$,
which are finite and surjective over a connected component of $\A^1_S$ and such that the normalization $\tilde{Z}$ of the closure of $Z$ in $\P^1\times S\times X$ 
satisfies $\infty_{S|\tilde{Z}}\ge D_{|\tilde{Z}}$, we have 
\[[Z_0]^*a=[Z_1]^*a,\]
where $[Z_\epsilon]$ denotes the finite correspondence from $S$ to $U$ associated to $Z\cap (\{\epsilon\}\times S)$, $\epsilon\in \{0,1\}\subset \A^1$.

A {\em reciprocity sheaf} in the sense of \cite{KSY2} is a presheaf with transfers $F$ which is a Nisnevich sheaf on $\Sm$ and 
for which any section $a\in F(U)$ has a modulus $(X,D)$. We denote by $\RSC_{\Nis}$ the category of reciprocity sheaves. 
For a proper modulus pair $(X,D)$ with $X\setminus D=U$ we set 
\[\tF(X,D)=\{a\in F(U)\mid (X,D) \text{ is a modulus for }a\}.\]
If $(X,D)$ is not proper we set 
\[\tF(X,D)=\varinjlim_{(Y,E)} \tF(Y,E),\]
where the limit is over the cofiltered ordered set of compactifications $(Y,E)$ of $(X,D)$, see \cite[1.8]{KMSY1}.
We also regularly work with pairs $(X,D)$, which are equal to a projective  limit $\varprojlim_{i\in I} (X_i, D_i)$ with $(X_i,D_i)$ modulus pairs and $I$ some filtered set
(e.g.,  $X$ is of finite type over a function field $K/k$ , $D$ is an effective Cartier divisor on $X$ and $U=X\setminus D$ is regular). 
In this case we set 
\[\tF(X,D)= \varinjlim_{i\in I} \tF(X_i,D_i).\]

\end{para}

\begin{para}\label{para:pairing}
Let $F\in \RSC_\Nis$. 
Let  $K$ be a function field  over $k$ and $U$ a regular quasi-projective $K$-scheme of dimension $d$.
Choose a factorization
\[U\xr{j} X\xr{\ol{f}} \Spec K\]
 of the structure map $U\to \Spec K$ with $X$ reduced, $j$ open dense, and $\ol{f}$ projective.
Building on the results from \cite{BRS}, we define in \cite[4.]{RS-ZNP} for such a factorization, a pushforward map
\[(\ol{f}, j)_*: H^d(X_\Nis, j_!(F\la d\ra_{U}))\to F(K),\]
where ``$j_!$'' denotes the extension-by-zero functor.  
Here $F\la d\ra\in \RSC_\Nis$ is the $d$th twist of $F$ introduced in \cite[5.5]{RSY} and $F\la d\ra_U$ denotes its restriction to $U_\Nis$.
There is a natural map of Nisnevich sheaves 
$F_U\otimes_\Z K^M_{d,U}\to  F\la d\ra_U$ on $U$, where $K^M_d$ denotes the Nisnevich sheafification of 
the improved Milnor $K$-theory from \cite{Kerz},
which  induces a morphism on $X_{\Nis}$
\eq{para:pairing1}{j_*F_U\otimes_\Z j_!K^M_{d,U}\to  j_!F\la d\ra_U.}
This yields the pairing
\ml{para:pairing2}{(-,-)_{U\subset X/K}: F(U)\otimes_{\Z} H^d(X_{\Nis}, j_!K^M_{d,U})\xr{\cup} H^d(X_{\Nis}, j_*F_U\otimes_{\Z} j_!K^M_{d,U})\\
\xr{\eqref{para:pairing1}} H^d(X_{\Nis}, j_!F\la d\ra_U)\xr{(\ol{f}, j)_*} F(K). }
It is a factorization of the usual pairing induced by finite correspondences from $\Spec K$ to $U$ in the
following sense: If $x\in U$ is a closed point there is a natural isomorphism
\eq{para:pairing2.5}{\theta_x:\Z\xr{\simeq} H^d_x(U_{\Nis}, K^M_{d,U}),}
induced by the Gersten resolution (see \cite[Proposition 10, (8)]{Kerz}). 
Composing with the natural map $H^d_x(U_{\Nis}, K^M_{d,U})\to H^d(X_{\Nis}, j_!K^M_{d,U})$ 
and taking the  sum over all closed points $x$ yields the map
\eq{para:pairing2.6}{ Z_0(U)=\bigoplus_{x\in U_{(0)}}\Z \to H^d(X_{\Nis}, j_!K^M_{d,U}).}
By \cite[Lemma 6.6]{RS-ZNP} we have  for all $a\in F(U)$ and $\zeta\in Z_0(U)$
\eq{para:pairing3}{(a, [\zeta])_{U\subset X/K}= \zeta^*a,}
where $[\zeta]$ on the  left denotes the image of $\zeta$ under \eqref{para:pairing2.6} and  
on the right we view $\zeta$ as a finite correspondence from $\Spec K$ to $U$.
\end{para}

\begin{para}\label{para:relK}
Let $X$ be a reduced noetherian excellent separated scheme of dimension $d<\infty$ over a field, 
such that $X^{(d)}=X_{(0)}$. Let $D\subset X$ be a nowhere dense closed subscheme. 
We define for $r\ge 1$
\[V_{r,X|D}=\Im(\sO^\times_{X|D}\otimes_{\Z} K^M_{r-1,X}\to K^M_{r, X}), \quad 
\text{where } \sO_{X|D}^\times=\Ker(\sO_X^\times \to \sO_D^\times).\]
This sheaf is very close to the relative Milnor $K$-sheaf $K^M_r(\sO_X, I_D)$ defined in \cite[(1.3)]{KS-GCFT}, where $I_D$ denotes the ideal sheaf of $D$.
In fact the two sheaves agree for $r=1$ and they have the same  stalks at all points with infinite residue field.
Since by Grothendieck-Nisnevich vanishing the cohomological dimension of the Nisnevich cohomology on a noetherian scheme is bounded by its dimension, see \cite[(1.2.5)]{KS-GCFT} or \cite[1.32 Theorem]{Nis}, we find that  the natural inclusion $V_{d, X|D}\inj K^M_d(\sO_X, I_D)$ induces an isomorphism
\eq{para:mod-pairing0.5}{ H^d(X_{\Nis}, V_{d, X|D})\xr{\simeq} H^d(X_{\Nis}, K^M_{d}(\sO_X,I_D)).}
For any regular dense open $j':U'\inj X$ contained in $X\setminus D$ the composition 
\eq{para:mod-pairing0.1}{
Z_0(U') \xr{\eqref{para:pairing2.6}} H^d(X_{\Nis}, j'_!K^M_{d,U'})\to H^d(X_{\Nis}, V_{d, X|D})} 
is therefore surjective by \cite[Theorem 2.5]{KS-GCFT}.
\footnote{Note that  as long as we work over a field the ``nice schemes'' in \cite[Theorem 2.5]{KS-GCFT} 
can be replaced by  excellent regular schemes, as follows from \cite[Proposition 10]{Kerz} 
and the comment below \cite[Corollary 2.4]{KS-GCFT}.}
\end{para}

\begin{para}\label{para:mod-pairing}
We recall the main result from \cite{RS-ZNP}. Let $F\in \RSC_{\Nis}$.
Let $K$ be a function field over $k$. Let $X$ be an integral projective $K$-scheme  of dimension $d$ 
and $j:U\inj X$  a regular dense open subscheme. 
Let $D\subset X$ be a closed subscheme (not necessarily a divisor) such that $D_{\red}=X\setminus U$.
Let $\nu: Y\to X$ be the normalization of $X$.
We define 
\eq{para:mod-pairing0}{F_{\gen}(X,D):=\Ker\left(F(U)\to \bigoplus_{y\in Y^{(1)}\cap \nu^{-1}(D)}\frac{F(Y_{y}^h\setminus y)}{\tF(Y_y^h, D_{y}^h)}\right),}
where $Y_y^h=\Spec \sO_{Y,y}^h$ and $D_{y}^h= D\times_X Y^h_y$.
We define $R(X|D)$ by the  exact sequence
\[0\to R(X|D)\to H^d(X_{\Nis}, j_!K^M_{d,U})\to H^d(X_{\Nis}, V_{d,X|D})\to 0.\]
\begin{thm}[{\cite[Proposition 6.7, Theorem 6.8]{RS-ZNP}}]\label{thm:mod-pairing}
Assumptions as in  \ref{para:mod-pairing}.
\begin{enumerate}[label=(\arabic*)]
\item\label{thm:mod-pairing1}  The pairing \eqref{para:pairing2} induces  a pairing
\[ (-,-)_{(X,D)/K}:F_{\gen}(X,D)\otimes_{\Z} H^d(X_{\Nis}, V_{d, X|D})\to F(K).\]
\item\label{thm:mod-pairing2}  Assume $X\in \Sm$ is projective over $k$ and $D$ is an effective Cartier divisor with simple normal crossing support.
Then for $a\in F(U)$ with $U=X\setminus D$ we have
\[\tF(X,D)=F_{\gen}(X,D)=
\{a\in F(U)\mid (a_K,\gamma)_{U_K\subset X_K/K}=0\; \forall K/k, \gamma\in R(X_K|D_K)\},\]
where in the set on the right, $K$ runs over all function fields over $k$, $X_K=X\otimes_k K$, and $a_K$ is the pullback of $a$ to $F(U_K)$.
\end{enumerate}
\end{thm}
In this paper we give a purely local description of the right hand side in \ref{thm:mod-pairing2},
using Par\v{s}in chains and higher local rings.
\end{para}

\section{Recollections on {Par\v{s}in} chains, higher local rings, and cohomology}
We recall some definitions  and results from \cite[(1.6)]{KS-GCFT} (see also \cite[5.]{RS-ZNP}).
In this section,  $X$ is a reduced noetherian separated scheme of dimension $d<\infty$, such that $X^{(d)}=X_{(0)}$.

\begin{para}\label{para:chain}
For $x, y\in X$ we write
\[y<x:\Longleftrightarrow \ol{\{y\}}\subsetneq \ol{\{x\}}, \text{ i.e., } y\in \ol{\{x\}} \text{ and } y\neq x.\]
A {\em chain } on $X$ is a sequence
\eq{para:chain1}{\ux=(x_0,\ldots, x_n)\quad \text{with } x_0<x_1<\ldots <x_n.}
The chain $\ux$ is a {\em maximal Par\v{s}in chain} (or  {\em maximal chain}) if 
$n=d$ and $x_i\in X_{(i)}$. Note that the assumptions on $X$ imply $x_{i}\in \ol{\{x_{i+1}\}}^{(1)}$.
We denote 
\[\c(X)=\{\text{chains on } X\}\quad \text{and} \quad \mc(X)=\{\text{maximal chains on } X\}.\]

A {\em maximal chain with break at $r\in\{0,\ldots, d\}$} is a chain \eqref{para:chain1}
with $n=d-1$ and $x_i\in X_{(i)}$, for $i<r$, and $x_i\in X_{(i+1)}$, for $i\ge r$.
We denote
\[\mc_r(X)=\{\text{maximal chain with break at $r$ on $X$}\}.\]
For $\ux=(x_0,\ldots, x_{d-1})\in \mc_r(X)$, we denote by $\b(\ux)$ the set of $y\in X_{(r)}$ such that
\eq{para:chain2}{\ux(y):=(x_0,\ldots, x_{r-1}, y, x_{r},\ldots, x_{d-1})\in \mc(X).}
\end{para}

\begin{para}\label{para:hlr}
Let $S\subset X$ be a finite subset contained in an affine open neighborhood of $X$.
A {\em strict Nisnevich neighborhood of $S$} is an \'etale map $u:U\to X $ such that $U$ is affine,  
the base change $u^{-1}(S)\to S$ of $u$ is an isomorphism, and every connected component of $U$ intersects 
$u^{-1}(S)$. 

Let $\ux=(x_0,\ldots, x_n)$ be a chain on $X$. A {\em strict Nisnevich neighborhood of $\ux$} 
is a sequence of maps
\[\sU=(U_n\to\ldots \to U_1\to U_0 \to X=:U_{-1}),\]
such that $U_i\to U_{i-1}$ is a strict Nisnevich neighborhood of 
$U_{i-1, x_i}= U_{i-1}\times_X \{x_i\}$, for all $i=0,\ldots, n$.
There is an obvious notion of morphism between two strict Nisnevich neighborhoods 
and picking a representative in each isomorphism class yields a filtered set
\[N(\ux):=\{\text{strict Nisnevich neighborhoods of }\ux\}.\]
Assume $\ux\in \mc_r(X)$ and $y\in \b(\ux)$. If $\sU$ is a strict Nisnevich neighborhood as above, then repeating $U_{r-1}$ in the $r$th spot
yields a map of filtered sets
\eq{para:hlr1}{N(\ux)\to N(\ux(y)).}

The Nisnevich stalk of a presheaf $F$ on $X_{\Nis}$ at $\ux\in \c(X)$ is defined to be 
\eq{para:hls.Nisnevichstalk}{
\Fh_{\ux}:=\varinjlim_{\sU=(U_n\to\ldots\to X)\in N(\ux)} F(U_n).}
Note that for $\ux\in \mc_r(X)$ and $y\in \b(\ux)$ the map \eqref{para:hlr1}
induces a natural map
\eq{para:hlr2}{\iota_y : \Fh_{\ux}\to \Fh_{\ux(y)}.}
For $F=\sO_X$ we write $\sOh_{X,\ul{x}}=\Fh_{\ux}$ and $K^h_{X,\ux}=\Frac(\sOh_{X,\ux})$.
By \cite[Lemma 5.3]{RS-ZNP} $\sOh_{X,\ux}=R_n$, where we  recursively define:
\[R_0=\sO_{X,x_0}^h \quad \text{and}\quad  \quad R_{i}=\prod_{\fp\in T_i} R_{i-1,\fp}^h,\quad i\ge 1,\]
where $T_i= \Spec R_{i-1}\times_X \{x_i\}$
is the finite set of prime ideals in $R_{i-1}$ lying over the prime ideal  in $\sO_{X,x_0}$ corresponding to $x_i$;
this is also the definition used in \cite{KS-GCFT}.
\end{para}

\begin{lem}\label{lem:hlr}
Let $\ux=(x_0,\ldots, x_n)\in\c(X)$.
Let $Y=\ol{\{x_n\}}$ and set $\uy=\ux$ viewed as a chain on $Y$. Then 
$\sOh_{X,\ux}/\fr = \sOh_{Y,\uy}$, where $\fr$ denotes the radical of $\sOh_{X,\ux}$. 
\end{lem}
\begin{proof}
If $\sU\in N(\ux)$, then 
$\sU\times_X Y\in N(\uy)$ and it follows from \cite[Proposition (18.6.8)]{EGAIV4},
that Nisnevich neighborhoods of the form $\sU\times_X Y$ are cofinal in $N(\uy)$. 
\end{proof}

\begin{para}\label{para:hlc}
Let $F$ be an abelian Nisnevich sheaf on $X$ and  $\ux=(x_0,\ldots, x_n)\in \c(X)$. We set 
\[H^i_{\ux}(X, F):= \varinjlim_{\sU=(U_n\to\ldots\to X)\in N(\ux)} H^i_{U_{n,x_n}}(U_{n,\Nis}, F),\]
where on the right hand side we consider the local cohomology group in the finite set $U_{n,x_n}=U_n\times_X \{x_n\}$.
Assume $x_{n-1}\in \ol{\{x_n\}}^{(1)}$ and write $\ux'=(x_0,\ldots, x_{n-1})$. There is a natural map
\eq{para:hlc1}{ \delta_{\ux}: H^i_{\ux}(X, F)\to H^{i+1}_{\ux'}(X,F)}
induced by the connecting homomorphism of  the localization sequence, see \cite[Definition 1.6.2(4)]{KS-GCFT}  (or \cite[5.4]{RS-ZNP}).
Following  \cite[Definition 1.6.2(5)]{KS-GCFT}, we define for a maximal chain $\ux=(x_0,\ldots, x_d)$ the map
\eq{para:hlc2}{c_{\ux}:=s_{x_0}\circ c_{\ux,0} : F(K^h_{X,\ux}):= \Fh_{\ux}\to H^d(X_{\Nis}, F),}
where $s_{x_0}: H^d_{x_0}(X_{\Nis}, F)\to H^d(X_{\Nis}, F)$ is the forget-support-map
and $c_{\ux, 0}$ is the composition
\ml{para:hlc3}{
c_{\ux, 0} : \Fh_{\ux}=H^0_{(x_0,\ldots, x_d)}(X_{\Nis}, F)
\xr{\delta_{(x_0,\ldots, x_d)}}H^1_{(x_0,\ldots x_{d-1})}(X_{\Nis}, F)
\xr{\delta_{(x_0,\ldots x_{d-1})}}\ldots \\
\ldots\xr{\delta_{(x_0,x_1)}} H^d_{x_0}(X_{\Nis}, F).}
\end{para}

\begin{prop}[{\cite[Lemma 1.6.3]{KS-GCFT}}]\label{prop:mapNisCoh}
For any abelian group $A$, the map
\[\Phi: \Hom(H^d(X_\Nis, F), A)\to \prod_{\ux\in \mc(X)} \Hom(\Fh_{\ux}, A), \quad 
\alpha\mapsto (\alpha\circ c_{\ux})_{\ux\in \mc(X)},\]
is injective. Furthermore, the image of $\Phi$ consists of  those tuples $(\chi_\ux)_{\ux\in\mc(X)}$
satisfying the following condition:
For any $r\in\{0,\ldots, d\}$,  $\ux\in \mc_r(X)$,  and for any $a\in \Fh_{\ux}$,
we have $\chi_{\ux(y)}(\iota_y (a))=0$ for almost all $y\in \b(\ux)$, and 
\eq{prop:mapNisCoh1}{\sum_{y\in \b(\ux)} \chi_{\ux(y)}(\iota_y(a))=0,}
where $\iota_y: \Fh_{\ux}\to \Fh_{\ux(y)}$ is the map from \eqref{para:hlr2}.
\end{prop}

\begin{para}
For $\ux\in \mc_r(X)$, $r\in \{0,\ldots, d\}$, and $a\in \Fh_{\ux}$, Proposition \ref{prop:mapNisCoh}
implies
\eq{para:hlc3.5}{c_{\ux(y)}(\iota_{y}(a))=0, \text{ for almost all }y\in \b(\ux) \quad \text{and}\quad
\sum_{y\in \b(\ux)}c_{\ux(y)}(\iota_{y}(a))=0.}
Note that in case $r=d$,  the composition 
\eq{para:hlc4}{\Fh_{\ux}\xr{\iota_y} \Fh_{\ux(y)}\xr{\delta_{\ux(y)}}
H^1_{\ux}(X_{\Nis}, F)}
is zero, for all $y\in\b(\ux)$. In particular $c_{\ux(y)}\circ \iota_y=0$, for all $y\in\b(\ux)$.
\end{para}

\def\cZar{c^{\Zar}}

\begin{para}\label{prop:mapNisCoh;rem1}
Let $F$ be a presheaf of abelian groups on $X_\Zar$ and $\ux=(x_0,\ldots, x_n)\in \c(X)$.
We can define the Zariski stalk $F_{\ux}$ of $F$ at $\ux$ as above, but in fact
$F_{\ux}=F_{x_n}$.
If $\ux\in \mc(X)$, we can define also the map analogous to \eqref{para:hlc3}:
\eq{para:hlc3Zar}{
\cZar_{\ux, 0} : F_{\ux} \to H^d_{x_0}(X_{\Zar}, F).}
In \cite{KS-GCFT} Proposition \ref{prop:mapNisCoh} is deduced by induction on the dimension from the coniveau  spectral sequence for Nisnevich cohomology and 
the Grothendieck-Nisnevich vanishing. Since the Zariski analogue of both statements hold, we also have a Zariski analogue of Proposition \ref{prop:mapNisCoh}. 
In particular, for $w\in X_{(0)}$, the map
\[\underset{\ux\in \mc(X),x_0=w}{\bigoplus} F_{\ux} \overset{\cZar_{\ux,0}}{\longrightarrow}
H^d_w(X_\Zar, F)\]
is surjective. This follows from the Zariski analogue of Proposition \ref{prop:mapNisCoh}
applied to the $(d-1)$-dimensional scheme $\Spec(\sO_{X,w})\backslash \{w\}$.
\end{para}

\section{Some auxiliary results for relative Milnor $K$-theory}


In this section $k$ denotes any field and 
$X$ is a reduced noetherian excellent separated $k$-scheme of dimension $d<\infty$, such that $X^{(d)}=X_{(0)}$.

\begin{para}\label{para:hltame}
Let $T$ be a noetherian reduced purely 1-dimensional and excellent semilocal scheme with total ring of fractions 
$\kappa(T)$ and denote by $\nu : \tilde{T}\to T$ the normalization.
Writing $T$ as a union of irreducible components $T=\cup_i T_i$ we obtain
$\kappa(T)=\prod_i \kappa(T_i)$ and $\tilde{T}=\coprod_i \tilde{T}_i$ with the obvious notation.
 Let $S$ be the set of closed points of $T$ and set $\kappa(S)=\prod_{s\in S} \kappa(s)$. 
Then we define
\[\partial_{S}:= \sum_{s\in S} \sum_{s'\in \nu^{-1}(s)}  \Nm_{\kappa(s')/\kappa(s)}\circ \partial_{v_{s'}} 
: K^M_r(\kappa(T))\to K^M_{r-1}(\kappa(S)),\]
where $v_{s'}$ denotes the discrete valuation on $\Frac(\sO_{\tilde{T},s'})$ defined by $s'$,
  $\partial_{v_{s'}}$ is the classical tame symbol, and $ \Nm_{\kappa(s')/\kappa(s)}$ is the norm map.

Let $\ux=(x_1, \ldots, x_{n-1}, x_n)\in \c(X)$ and assume $x_{n-1}\in X_{(d-1)}$ and $x_n\in X_{(d)}$,
in particular $x_{n-1}\in \ol{\{x_n\}}^{(1)}$.
Set $X'=\ol{\{x_{n-1}\}}$ and  $\ux':=(x_1, \ldots, x_{n-1})\in\c(X')$.
We define
\eq{para:hltame0}{\partial_{\ux}:=\partial_{S_{d-1}}: K^M_r(K^h_{X,\ux})\to K^M_r(K^h_{X',\ux'})}
Here we use the following notation: 
$S_{d-1}$ is the set of closed points of the  reduced 1-dimensional and excellent semi-local ring  $\sO^h_{X,\ux'}$, 
where we view  $\ux'$ as a chain  on $X$;
note $K^h_{X,\ux}=\Frac(\sO^h_{X,\ux'})$ and  $\kappa(S_{d-1})=K^h_{X',\ux'}$, by Lemma \ref{lem:hlr}.

For $\ux\in \mc(X)$ as above and $i\in\{0,\ldots, d\}$ set
$\ux_i=(x_0,\ldots, x_i)\in\mc(\ol{\{x_i\}})$.
We denote by $\partial_{X, \ul{x}}$ the following composition
\eq{para:hltame1}{\partial_{X, \ul{x}}:=
\partial_{\ux_1}\circ \partial_{\ux_2}\circ\ldots\circ \partial_{\ux_d}
: K^M_r(K^h_{X,\ux})\to K^M_{r-d}(\kappa(x_0)). }
For $d=0$ this is  the identity (by convention).
\end{para}

\begin{lem}\label{lem:c-vs-del}
Assume $x_0$ is contained in $X_{\rm reg}$ the regular locus of $X$.
Then  the following diagram commutes
\[\xymatrix{
K^M_d(K^h_{X,\ux})\ar[rr]^{c_{\ux,0}}\ar[dr]_{\partial_{X,\ux}}& &
H^d_{x_0}(X_{\Nis}, K^M_{d,X}),\\
& \Z\ar[ur]_-{\theta_{x_0}}^-{\simeq}
}\]
where $c_{\ux,0}$ is the map  \eqref{para:hlc3}, $\partial_{X,\ux}$ is the map \eqref{para:hltame1}, and
$\theta_{x_0}$ the  map \eqref{para:pairing2.5}.
\end{lem}
\begin{proof}
We may assume $X=X_{\rm reg}$.
By \cite[Proposition 10(8)]{Kerz} the Gersten complex viewed as complex on $X_\Nis$ is a resolution of $K^M_{d,X}$;
since its terms are furthermore acyclic for the global section functor,  we may use it to compute 
the local cohomology. This in particular yields the identifications in the diagram below for $0\le j\le d-1$
\eq{lem:c-vs-del1}{\xymatrix{
H^j_{(x_0,\ldots, x_{d-j})}(X_{\Nis}, K^M_{d,X})\ar[rr]^-{\delta_{(x_0,\ldots, x_{d-j})}}\ar@{=}[d]& &
H^{j+1}_{(x_0,\ldots, x_{d-j-1})}(X_{\Nis}, K^M_{d,X})\ar@{=}[d]\\
K^M_{d-j}(K^h_{Y, (x_0,\ldots, x_{d-j})})\ar[rr]^-{\partial_{S_{d-j-1}}} & &
K^M_{d-j-1}(K^h_{Z, (x_0,\ldots, x_{d-j-1})}),
}}
where $Y:=\ol{\{x_{d-j}\}}$, $Z:=\ol{\{x_{d-j-1}\}}$, and  $S_{d-j-1}$ denotes the set of closed points in 
$\sOh_{Y, (x_0,\ldots, x_{d-j-1})}$ and the other notation is taken from \ref{para:hlc} and \ref{para:hltame}. 
Composing the diagrams for $j=0,\ldots, d-1$ yields the statement. 
\end{proof}

\begin{lem}[{\cite[Proposition 2.9]{KS-GCFT}}]\label{lem4;recpairing}
Let $i: Y\to X$ be a closed immersion with $Y$ integral of dimension $e$ and assume  $Y\cap X_{\rm reg}\neq \emptyset$.
Let $D\subset X$ be a closed subscheme which does not contain $Y$.
Then there exists a proper closed subscheme $E\subset Y$ and a map (see \ref{para:mod-pairing} for notation)
\eq{gysinKM}{ i_*: H^e(Y_\Nis,V_{e,Y|E}) \to H^d(X_\Nis,V_{d, X|D})}
which is uniquely determined by the requirement that 
for any regular open $U\subset X\setminus D$ and regular open $U'\subset (Y\cap U)\setminus E$ which is dense in $Y$, 
the following diagram commutes:
\eq{gysinKM2}{\xymatrix{
 Z_0(U') \ar[r]\ar[d]^{i_*} & H^e(Y_\Nis,V_{e, Y|E}) \ar[d]^{i_*} \\
Z_0(U)\ar[r]&  H^d(X_\Nis,V_{d, X|D})}}
where the horizontal maps are the maps \eqref{para:mod-pairing0.1}.
Moreover, for any $\ul{y}\in \mc(Y)$ and any $\ul{x}=(\uy, x_{e+1},\ldots, x_d)\in \mc(X)$, 
the following diagram is commutative:
\eq{gysinKM3}{\xymatrix{
K^M_{e}(K^h_{Y,\ul{y}}) \ar[r]^-{c_{\ul{y}}} & H^{e}(Y_\Nis, V_{e, Y|E}) \ar[d]^{i_*} \\
K^M_{d}(K^h_{X,\ul{x}}) \ar[r]^-{c_{\ul{x}}}\ar[u]^{\partial^{\ux}_{\uy}} & H^d(X_\Nis, V_{d, X|D})}}
where we set (using the notation from \ref{para:hltame})
\[\partial^{\ux}_{\uy}:=\partial_{\ux_{e+1}}\circ\ldots\circ\partial_{\ux_d} : 
K^M_d(K^h_{X,\ux}) \to K^M_{e}(K^h_{Y,\uy}),\]
with $\ux_j=(\uy, x_{e+1},\ldots, x_j)$, for $j\ge e+1$.
\end{lem}
\begin{proof}
This is essentially \cite[Proposition 2.9]{KS-GCFT} and the same proof works.
Since in {\em loc. cit.} the assumptions  and  the  formulation are  slightly different, 
we repeat the argument for the convenience of the reader.
First note that if a map $i_*$ as in the statement exists such that \eqref{gysinKM3} commutes, then 
\eqref{gysinKM2} commutes as well, by Lemma \ref{lem:c-vs-del}. This later commutativity uniquely characterizes
the map $i_*$, as the horizontal maps in \eqref{gysinKM2} are surjective, see \eqref{para:mod-pairing0.1}.
Thus it remains to construct a map $i_*$ (for an appropriate $E$) which makes \eqref{gysinKM3} commutative.
By \cite[Proposition 2.7]{KS-GCFT} there exists a proper closed subscheme $E\subset Y$ with ideal sheaf $I_E$, such that
for any \'etale map $X'\to X$ and any $y'\in X'$ over the generic point of $Y$ with closure $Y'=\ol{\{y'\}}$ and any 
$w\in {Y'}^{(1)}$ the composition 
\mlnl{K^M_e(\sO_{Y'}, I_E\sO_{Y'})_w\to K^M_{e}(k(y'))\cong H^{d-e}_{y'}(X'_{\Nis}, K^M_d(\sO_X, I_D))\\
\to H^{d-e}_{(w,y')}(X'_{\Nis}, K^M_d(\sO_X, I_D))\xr{\delta_{(w,y')}}
H^{d-e+1}_w(X'_{\Nis}, K^M_d(\sO_X, I_D))}
is zero. We fix this $E$ in the following. For $\uy=(y_0,\ldots, y_e)\in\mc(Y)$ we consider the composition
\mlnl{\chi_{\uy}: (V_{e, Y|E})_{\uy}^h=K^M_e(K^h_{Y,\uy})\xr{\simeq} H^{d-e}_{\uy}(X_{\Nis}, V_{d, X|D})\\
\xr{\delta_1\circ\ldots\circ \delta_e} H^d_{y_0}(X_{\Nis}, V_{d,X|D}) \to H^d(X_{\Nis}, V_{d, X|D}),}
where the first isomorphism is induced by the Gersten resolution 
together with the definition of the local cohomology group in \ref{para:hlc} and we set 
$\delta_i=\delta_{(y_0,\ldots, y_i)}$ with the notation from \eqref{para:hlc1}.
The family $\{\chi_{\uy}\}_{\uy\in \mc(Y)}$ satisfies the condition \eqref{prop:mapNisCoh1}: for
$r\in \{0,\ldots, e-1\}$ this follows from the definition of the $\delta_i$, for $r=e$ it follows from our choice of $E$ above and 
\eqref{para:mod-pairing0.5}.
Thus by Proposition \ref{prop:mapNisCoh} there is a map $i_*$ as in the statement, 
such that 
\[ i_*\circ c_{\uy}=\chi_{\uy}:K^M_e(K_{Y,\uy}^h)\to H^d(X_{\Nis}, V_{d,X|D}),
\]
 all $\uy\in \mc(Y)$. The commutativity of \eqref{gysinKM3} follows immediately from this equality together with
 the commutativity of \eqref{lem:c-vs-del1}. 
\end{proof}

\begin{para}\label{para:Vbar}
We recall some constructions and results from \cite[\S4]{KS-GCFT}.
Let $D$ be a closed subscheme of $X$ which is nowhere dense and is defined by the ideal sheaf $I\subset \sO_X$. 
We define the Nisnevich sheaf $\Vb_{r, X|D}$ on $X$ by 
\[U\mapsto \Vb_{r,X|D}(U)=
\Ker\left(\bigoplus_{\eta\in U^{(0)}} K^M_r(k(\eta))\to \bigoplus_{x\in U^{(1)}}
\frac{\oplus_{\eta>x}K^M_r(k(\eta)^h_x)}{(V_{r, U|D_U})_x^h}\right),\]
where $U$ runs over the \'etale $X$-schemes and $D_U=D\times_X U$. 
Note that this sheaf agrees with the sheaf $\ol{K}^M_r(\sO_X, I)$ defined in {\em loc. cit.} for $r=1$ and, if $d\ge 2$, for all $r\ge 1$: 
For $r=1$, this is immediate. For $d\ge 2$, $(V_{r, U|D_U})_x^h=K^M_r(\sO_U, I_U)^h_x$ since the residue fields $k(x)$, for $x\in U^{(1)}$, have infinitely many elements (see \ref{para:mod-pairing}). 
Note that we have a natural map
\[V_{r,X|D}\to \Vb_{r, X|D}.\]
The cokernel  of this map is supported in codimension 2 and
the kernel in codimension 1.  Hence Grothendieck-Nisnevich vanishing yields an isomorphism
\eq{para:Vbar0}{H^d(X_\Nis, V_{r, X|D})\xr{\simeq} H^d(X_\Nis, \Vb_{r, X|D}).}
We will need the following  statement from {\em loc. cit.}:

\end{para}

\begin{prop}[{\cite[Proposition 4.2]{KS-GCFT}}]\label{para:Vbar2}
Let $f: Y\to X$ be a finite morphism and assume $Y$ is reduced and $f(Y^{(0)})\subset X^{(0)}$. 
Assume $r=1$ or $d\ge 2$. Then the 
norm map on Milnor $K$-theory
\[\Nm: f_*\bigg(\bigoplus_{\eta\in  Y^{(0)}} i_{\eta_*}K^M_{r,\eta}\bigg) \to 
      \bigoplus_{\xi\in  X^{(0)}} i_{\xi_*}K^M_{r,\xi},\]
where $i_\eta:\eta\inj Y$ and $i_\xi: \xi\inj X$ are the natural inclusions,
induces a morphism
\[\Nm: f_*(\Vb_{r, Y|E})\to \Vb_{r, X|D},\]
for some large enough nowhere dense closed subscheme $E\subset Y$ containing $D\times_X Y$.
\end{prop}

\begin{cor}\label{cor:approx}
Let $f:Y\to X$ be a separated morphism of finite-type with $f(Y^{(0)})\subset X^{(0)}$
and $\dim Y=\dim X=d$. Assume $r=1$ or $d\ge 2$ and $r\ge 1 $.
Let $D\subset X$ be a closed subscheme and set  $D_Y:=D\times_X Y$.
Let $x_1,\ldots, x_n\in X^{(1)}$ and $\beta_1,\ldots, \beta_n\in \bigoplus_{\eta\in Y^{(0)}}K^M_r(k(\eta))$.
Assume that there exists an open affine subscheme in $X$ containing all the points $x_1,\ldots, x_n$.
Then there exits an element $\gamma\in \bigoplus_{\eta\in Y^{(0)}}K^M_r(k(\eta))$ such that for all $i=1,\ldots, n$
\[\gamma-\beta_i\in (\Vb_{r, Y|D_Y})_{y},\quad \text{for all }y\in f^{-1}(x_i),\]
\[\Nm(\gamma)-\Nm(\beta_i)\in (\Vb_{r, X|D})_{x_i}, \]
where $\Nm: \bigoplus_{\eta\in Y^{(0)}}K^M_r(k(\eta))\to \bigoplus_{\xi\in X^{(0)}}K^M_r(k(\xi))$ is the norm map.
\end{cor}
\begin{proof}
We may replace $f$ by a compactification and hence assume that $f$ is proper.
Furthermore, we may replace $X$ by its semi-localization at the points 
$x_1, \ldots x_n$ and $f$ by the base change. The semi-localization 
exists since $x_1,\ldots, x_n$ are contained in an affine open in $X$.
Thus $X$ is affine, integral, excellent, and 1-dimensional and $f:Y\to X$ is a 
proper, dominant, and quasi-finite morphism, whence it is finite and surjective. 
Let $\nu:\tY\to Y$ be the normalization. Thus $\tY$ is a finite disjoint union of Dedekind schemes.
By Proposition \ref{para:Vbar2} we find a closed subscheme 
$E\subset \tY$ containing $D_{\tY}$ such that 
\eq{cor:approx1}{\nu_*(\Vb_{r,\tY|E})\subset \Vb_{r, Y|D_Y} \quad \text{and}\quad 
 \Nm((f\circ\nu)_*(\Vb_{r,\tY|E}))\subset \Vb_{r,X|D}.}
By the Approximation Lemma, we find an element $\gamma\in \bigoplus_{\eta\in Y^{(0)}}K^M_r(k(\eta))$
such that
\[\gamma-\beta_i \in (\Vb_{r,\tY|E})_{\tx_i}\quad \text{for all } 
\tx_i\in (f\nu)^{-1}(x_i).\]  
The statement follows from this and \eqref{cor:approx1}.
\end{proof}

\section{Higher local symbols}\label{sec:HLS}
We introduce higher local symbols along maximal chains for  reciprocity sheaves. 
These generalize local symbols for curves, see \cite[Proposition 5.2.1]{KSY1} 
and \cite{Serre-GACC} for the classical case of commutative $k$-groups. Furthermore,
we obtain a unified construction for several higher local symbols defined in the literature, e.g., by Par\v{s}in, Lomadze, Kato
and many more. The results will be used in section \ref{sec:mod-LS} to give a 
characterization of the modulus in terms of local symbols. The content of this section will also
play a crucial role in \cite{RS-AS}.

In this section $k$ is a perfect field, 
$K$ is function field over $k$, and $X$ is an integral  scheme of finite type over $K$ and  dimension $d$.
We fix $F\in \RSCNis$. 

\begin{defn}\label{defn:HLS}
Let $\ux=(x_0,\ldots, x_d)\in \mc(X)$. We define the pairing
\eq{defn:HLS1}{(-,-)_{X/K,\ux}: F(K^h_{X,\ux})\otimes_{\Z} K^M_d(K^h_{X,\ux})\to F(K)}
as follows:
Choose  an open subscheme $V\subset X$ which is quasi-projective and contains $x_0$.
Choose a dense open regular subscheme $U\subset V$.
Choose a dense open immersions $j_V:V\inj Y$ into an integral projective $K$-scheme (a {\em projective compactification})
with structure map $f:Y\to \Spec K$
and denote by $j:U\inj Y$ the induced immersion. Note that $\ux\in\mc(Y)$.
We define \eqref{defn:HLS1} as the composition
\mlnl{F(K^h_{X,\ux})\otimes_{\Z} K^M_d(K^h_{X,\ux})=(F_{U}\otimes j_!K^M_{d,U})_{\ux}\to 
j_!(F\la d\ra_{U})_{\ux}\\
\xr{c_{\ux}} H^d(Y_{\Nis}, j_!F\la d\ra_{U} )\xr{(f,j)_*} F(K),}
where the first map is the stalk at $\ux$ of \eqref{para:pairing1}, $c_{\ux}$ is the map \eqref{para:hlc2},
and $(f,j)_*$ is the pushforward recalled in \ref{para:pairing}.
It follows from Lemma \ref{lem:HLS}\ref{lem:HLS1}, \ref{lem:HLS2} below
that this definition is independent of the choice of $V$, $U$,  and $Y$.

Take $r,s\in\{0,\ldots,d\}$. Precomposing \eqref{defn:HLS1} with the natural map 
\[F(K^h_{X, (x_{r},\ldots, x_d)})\otimes_{\Z} K^M_d(K^h_{X, (x_s, \ldots, x_d)})
\to F(K^h_{X,\ux})\otimes_{\Z} K^M_d(K^h_{X,\ux}),\]
cf. \eqref{para:hlr2}, we obtain pairings (denoted by the same symbol)
\eq{defn:HLS2}{(-,-)_{X/K,\ux}: 
F(K^h_{X,(x_{r}, \ldots, x_d)})\otimes_{\Z} K^M_d(K^h_{X,(x_{s}, \ldots, x_d)})\to F(K),}
in particular, for $r=d$ and $s=d-1$,  we get the pairing
\eq{defn:HLS3}{(-,-)_{X/K,\ux}: 
F(K(X))\otimes_{\Z} K^M_d(\Frac(\sOh_{X,x_{d-1}}))\to F(K).}

We call \eqref{defn:HLS1}, \eqref{defn:HLS2}, and \eqref{defn:HLS3} the {\em higher local symbol of $F$ at $\ux$}.
\end{defn}

\begin{lem}\label{lem:HLS}
Let the situation be as above.
\begin{enumerate}[label=(\arabic*)]
\item\label{lem:HLS1} 
The definition of the higher local symbol \eqref{defn:HLS1} is independent of the choice
of the quasi-projective open $V\subset X$ containing $x_0$, the regular dense open subset $U\subset V$,
 and the choice of the projective compactification $V\inj Y$.
\item\label{lem:HLS2} Let $X$ be quasi-projective,   $U\subset X$ a regular dense open, 
and $j: U\inj X\inj  \ol{X}$ a projective compactification. 
Let $a\in F(U)$ and $\beta\in K^M_d(K^h_{X,\ux})$.
Then
\[(a,\beta)_{X/K,\ux}=(a, \c_{\ux}(\beta))_{U\subset \ol{X}/K},\quad \text{for all }\ux\in\mc(X),\]
where $(-,-)_{U\subset \ol{X}/K}: F(U)\otimes_{\Z} H^d(\ol{X}_{\Nis}, j_!K^M_{d,U})\to F(K)$
is from \eqref{para:pairing2}.
\end{enumerate} 
\end{lem}
\begin{proof}
\ref{lem:HLS1}. We start by showing the independence of the choice of the projective compactification of $V$. 
Thus assume we have two projective compactifications $j: U\inj V\inj Y$ and $j':U\inj V \inj Y'$ and denote by 
$f:Y\to \Spec K$ and $f':Y'\to \Spec K$ the projective structure maps. 
It suffices to consider the situation, where we have a projective morphism $g:Y'\to Y$
such that $g\circ j'=j$ and $f\circ g=f'$. In this case the independence follows from the following commutative diagram
\[\xymatrix{
  F\la d\ra(K^h_{X,\ux}) \ar[r]^-{c_{\ux}} \ar[rd]_-{c_{\ux}} &
  H^d(Y_\Nis,j_! F\la d\ra_U) \ar[d]\ar[r]^-{(f, j)_*}  & F(K) \\
 &H^d(Y'_\Nis, j'_! F\la d\ra_U)\ar[ru]_-{(f',j')_*},}\]
where the vertical map in the middle is induced by the natural map $j_!\to Rg_* j'_!$, see \cite[(4.3.3)]{RS-ZNP}.
The right triangle commutes by \cite[Lemma 4.7(3)]{RS-ZNP}. The left triangle commutes since 
both maps labeled $c_{\ux}$ factor over $H^d_{x_0}(V, F\la d \ra_V)$.

Next we show the independence of the choice of $U$.
 By the above, it  suffices to consider a dense open immersion $\nu: U'\inj U$  with projective compactifications 
$j: U\inj V\inj   Y$ and $j'=j\circ\nu: U'\inj U\inj V\inj Y$. In this case the independence 
 follows from the commutative diagram 
\[\xymatrix{
 F\la d\ra(K^h_{X,\ux}) \ar[r]^-{c_{\ux}} \ar[rd]_-{c_{\ux}} &
 H^d(Y_\Nis,j'_! F\la d\ra_{U'}) \ar[d]_-{\rm nat.}\ar[r]^-{(f,j')_*}  & F(K) \\
& H^d(Y_\Nis,j_! F\la d\ra_U)\ar[ru]_-{(f,j)_*}.}\]
Here the commutativity of the right triangle holds by \cite[Lemma 4.7(2)]{RS-ZNP} and the one of the left triangle is obvious.

It remains to check the independence of the choice of $V$. To this end, let $V,V'\subset X$ be two open quasi-projective subschemes 
containing $x_0$,  let $V\inj Y\to \Spec K$  and  $V'\inj Y'\to \Spec K$ be projective compactifications,
and let $U\subset V$ and $U'\subset V'$ be two open regular subschemes.
We obtain the open immersions $U\inj Y$ and $U'\inj Y'$.
Let $V''\subset V\cap V'$ be an affine open neighborhood of $x_0$  and let $U''\subset U\cap U'\cap V''$ be a dense open regular subscheme.
We have two induced open immersions $U''\inj Y$ and $U'\inj Y'$.
Denote by $\eqref{defn:HLS1}_{(U,V,Y)}$ the pairing \eqref{defn:HLS1} constructed using $U\inj V\inj Y\to \Spec K$.
Then 
\[\eqref{defn:HLS1}_{(U,V,Y)}=\eqref{defn:HLS1}_{(U'',V,Y)}= \eqref{defn:HLS1}_{(U'',V'',Y)}=\eqref{defn:HLS1}_{(U'',V'',Y')},\]
where the first equality holds by the independence of the choice of $U$ proven above, the second equality holds by definition of the pairing (it only depends on the maps
$U\inj Y\to \Spec K$), and the third equality
holds by the independence of the choice of the compactification of $V''$ proven above. This together with the analog reasoning for
$(U',V',Y')$ instead of $(U,V,Y)$, implies $\eqref{defn:HLS1}_{(U,V,Y)}=\eqref{defn:HLS1}_{(U',V',Y')}$.

\ref{lem:HLS2}. This follows form the compatibility of the boundary maps from the localization sequence with
cup-products, see \cite[(6.7.5)]{RS-ZNP}
\end{proof}

\begin{prop}\label{prop:HLS}  
The pairings \eqref{defn:HLS1} satisfies  the following properties for all $a\in F(K(X))$:

\begin{enumerate}[label= (HS\arabic*)]
\item\label{HS1}
Let $X\hookrightarrow X'$ be an open immersion where $X'$ is an integral $K$-scheme of dimension $d$. Then for all $\beta\in K^M_d(K^h_{X,\ux})$
\[ (a,\beta)_{X/K,\ux}=(a,\beta)_{X'/K,\ux},\]

\item\label{HS2} Let $\ux=(x_0,\ldots, x_d)\in\mc(X)$, and $X_{d-1}\subset X$ be the closure 
of $x_{d-1}$, and set $\ux'=(x_0,\ldots, x_{d-1})\in \mc(X_{d-1})$.
Then for all $\beta\in K^M_d(K^h_{X,\ux})$
\[(a, \beta)_{X/K, \ux} =\begin{cases} 
      \beta\cdot \Tr_{K(X)/K}(a), & \text{if }d=0,\\
     (a(x_{d-1}), \partial_{\ux}\beta)_{X_{d-1}/K, \ux'}, & \text{if } d\ge 1 \text{ and } a\in F(\sO_{X, x_{d-1}}),
\end{cases}\]
where 
$a(x_{d-1})\in F(K(X_{d-1}))$ is the restriction of $a$ and $\partial_{\ux}$ is defined in \eqref{para:hltame0}, and $\Tr_{K(X)/K}: F(K(X))\to F(K)$ is the trace for the finite map $\Spec K(X)\to \Spec K$.

\item\label{HS3} Let $D\subset X$ be a closed subscheme such that $X\setminus D$ is nonempty and regular.
Assume $a\in F_{\gen}(X,D)$. Then, for $\ux=(x_0,\ldots, x_d)\in \mc(X)$ and $\ux'=(x_0,\ldots, x_{d-1})\in \mc_d(X)$,
we have (See \ref{para:Vbar} for the definition of $\Vb_{d, X|D}$) 
\[(a, \beta)_{X/K,\ux}=0,\quad \text{for all }  \beta\in  (\Vb_{d, X|D})^h_{\ux'}.\]

\item\label{HS4} 
Let $\ux'\in \mc_r(X)$ with $0\leq r\leq d-1$. Then for all $\beta\in K^M_d(K^h_{X,\ux'})$
\[(a, \iota_y\beta)_{X/K,  \ux'(y)}=0,\quad \text{for almost all } y\in \b(\ux').\]
If either $r\ge 1$ or $r=0$, $X$ is quasi-projective, and the closure of $x_1$ in $X$ is projective over $K$, where $\ux'=(x_1,\ldots, x_d)$, 
then
\[\sum_{y\in b(\ux')} (a, \iota_y\beta)_{X/K,  \ux'(y)}=0, \]
where $\iota_y: K^h_{X,\ux'} \to K^h_{X,{\ux'(y)}}$ is the map \eqref{para:hlr2}.
\end{enumerate}
Furthermore, the family 
\[\bigg\{(-,-)_{X/K,\ux}: F(K(X))\otimes K^M_{\dim X}(K(X))\to F(K)\mid \ux\in\mc(X)\bigg\}_{ \dim X\le d},\]
where $X$ is running through all integral schemes of finite type of dimension $\le d$,
is uniquely determined by \ref{HS1}-\ref{HS4}. 
\end{prop}
\begin{proof}
\ref{HS1} follows from Lemma \ref{lem:HLS}\ref{lem:HLS1}.
\ref{HS2}. For $d=0$ this follows from the fact that in this case the pushforward $(f,j)_*$  appearing in Definition 
\ref{defn:HLS} is the pushforward along the finite map $f:\Spec K(X)\to \Spec K$ constructed in \cite{BRS},
which is equal to the trace by \cite[Proposition 8.10(3)]{BRS}.
Now assume $d\ge 1$ and take $a\in  F(\sO_{X, x_{d-1}})$,  $\beta\in K^M_d(K^h_{X,\ux})$.
By \ref{HS1} we may assume that $X\to \Spec K$ is projective. 
We find a closed subscheme  $D\subset X$ with $x_{d-1}\not\in D$, such that $U=X\setminus D$ is regular
and $a\in F_{\gen}(X,D)$. Denote by  $i:Y:=X_{d-1}=\ol{\{x_{d-1}\}}\inj X$ the closed immersion.
Choose the closed subscheme $E\subset Y$ as in Lemma \ref{lem4;recpairing}
and choose a regular open  $U'\subset (Y\cap U)\setminus E$ which is dense in $Y$.
Since the map $Z_0(U')\to H^{d-1}(Y_\Nis, V_{d-1, Y|E})$ from \eqref{para:mod-pairing0.1} is surjective,
 we find a cycle $\zeta\in Z_0(U')$ which maps to $c_{\ux'}(\partial_\ux(\beta))$; we can view $\zeta$ as a finite correspondence from $\Spec K$ to $U'$.
We denote by 
\[[\zeta]=c_{\ux'}(\partial_\ux(\beta))\in H^{d-1}(Y, j'_!K^M_{d-1, U'})\quad \text{and}\quad  [i_*\zeta]\in H^{d}(X, j_! K^M_{d, U})\]
the images of $\zeta$ and $i_*\zeta\in Z_0(U)$ under the cycle maps \eqref{para:pairing2.6}. We compute
\begin{align*}
(a,\beta)_{X/K,\ux}&= (a,c_{\ux}(\beta))_{U\subset X/K}, & &\text{by Lem. \ref{lem:HLS}\ref{lem:HLS2},} \\
                      &= (a,c_{\ux}(\beta))_{(X,D)/K},& & \text{by Thm. \ref{thm:mod-pairing}\ref{thm:mod-pairing1},}\\
                      &= (a,i_*(c_{\ux'}\partial_{\ux}(\beta)))_{(X,D)/K}, & & \text{by \eqref{gysinKM3},}\\
                      &= (a,i_*[\zeta])_{U\subset X/K}, & & \text{by Thm. \ref{thm:mod-pairing}\ref{thm:mod-pairing1},}\\
                      &= (a, [i_*\zeta])_{U\subset X/K}, & & \text{by \eqref{gysinKM2}},\\
                      &=(i_*\zeta)^*a, & & \text{by \eqref{para:pairing3}},\\
                      &={\zeta}^*i^*a, & & \text{by defn of corr. action},\\
                      &= (a(x_{d-1}), [\zeta])_{U'\subset Y/K}, & & \text{by \eqref{para:pairing3}},\\
                      &=(a(x_{d-1}), \partial_{\ux}(\beta))_{Y/K,\ux'}, & &\text{by Lem. \ref{lem:HLS}\ref{lem:HLS2}.}
\end{align*}
This yields \ref{HS2}. Property \ref{HS3} follows from Lemma \ref{lem:HLS}\ref{lem:HLS2},  
Theorem \ref{thm:mod-pairing}\ref{thm:mod-pairing1}, \eqref{para:Vbar0}, and the vanishing of \eqref{para:hlc4}.
For \ref{HS4} in the case $r\ge 1$ choose a quasi-projective open $V\subset X$ which contains the closed point $x_0$ of $\ux'$, and take 
a projective compactification $V\inj Y$. Note that $\ux'$ also defines a chain on $Y$ and the set $\b(\ux')$ does not change when we consider $\ux'$ as a chain on 
$X$ or $Y$. Hence in this case the statement  follows directly from  Definition \ref{defn:HLS} and  \eqref{para:hlc3.5} applied to $F=j_!K^M_{d,U}$, 
where $j:U\inj Y$ is the inclusion of a dense open regular subscheme. 
In the case $r=0$, we can take a projective compactification $X\inj \ol{X}$ and  view $\ux'$ as a chain on $\ol{X}$.
By the assumption that the closure of $x_1$ in $X$ is projective, the set $\b(\ux')$ does not change when we consider 
$\ux'$ as a chain on $X$ or on $\ol{X}$. Hence the statement follows from $\eqref{para:hlc3.5}$ also in this case.

It remains to prove the uniqueness part. Let $\{\la-,-\ra_{X/K,\ux}\mid \ux\in\mc(X)\}_{\dim X\le d}$ be another
family of symbols satisfying \ref{HS1}-\ref{HS4}. By  \ref{HS1} it suffices to show 
$\la-,-\ra_{X/K,\ux}=(-,-)_{X/K,\ux}$ for all affine $K$-schemes $X$; 
applying \ref{HS1} again we may assume $X$ is projective.
We proceed by induction. If $\dim X=0$  the symbol is uniquely determined by \ref{HS2}.
Now we assume $\dim X=d\ge 1$ and 
the symbols coincide on all closed subschemes of $X$ of dimension strictly smaller than $d$.
Set $L:=K(X)$.
Let $a\in F(L)$, $\beta\in K^M_d(L)$, and $\ux=(x_0,\ldots, x_d)\in \mc(X)$. 
Let $D\subset X$ be a strict closed subscheme such that $X\setminus D$
is regular, $x_{d-1}\in D$, and $a\in F_{\gen}(X,D)$. 
By Corollary \ref{cor:approx} (with $f=\id$)
we find an element $\gamma\in K^M_d(L)$ such that 
$\beta-\gamma\in (\Vb_{d, X|D})_{x_{d-1}}$ and 
$\gamma\in (\Vb_{d, X|D})_u$, 
for all $u\in X^{(1)}\cap D\setminus\{x_{d-1}\}$.
Set $\ux'=(x_0,\ldots, x_{d-2}, x_d)\in \mc_{d-1}(X)$. Then
\mlnl{(a,\beta)_{X/K,\ux}\stackrel{\ref{HS3}}{=}\sum_{u\in \b(\ux')\cap D}(a,\gamma)_{X/K,\ux'(u)}
\stackrel{\ref{HS4}}{=}-\sum_{u\in \b(\ux')\setminus D} (a,\gamma)_{X/K,\ux'(u)}\\
\stackrel{\ref{HS2}}{=} 
-\sum_{u\in \b(\ux')\setminus D} (a(u),\partial_{\ux'(u)}\gamma)_{\ol{\{u\}}/K,(x_0,\ldots x_{d-2}, u)}.}
The same computation with $\la-,-\ra_{X/K,\ux}$ and induction yields the desired equality.
\end{proof}

The formulation of the above proposition was inspired by the treatment of local symbols in \cite[III, \S1]{Serre-GACC}.
But note that the construction is completely different. The next proposition is however
a formal consequence of \ref{HS1}-\ref{HS4} (and properties of Milnor $K$-theory)
in the same way \cite[III, Proposition 4]{Serre-GACC} is a consequence of the properties written in Definition 2 
of {\em loc. cit.}
\begin{prop}\label{prop:HS5}
Let $f: Y\to X$ be a projective and surjective  $K$-morphism between two integral $K$-schemes of the same dimension $d$.
Then we have for all $a\in F(K(X))$ and $\beta\in K^M_d(K(Y))$ and $\ux\in\mc(X)$
\[\tag*{(HS5)}\label{HS5} \sum_{\substack{\uy\in \mc(Y)\\ f(\uy)=\ux}}(f^*a, \beta)_{Y/K,\uy}=(a, f_*\beta)_{X/K,\ux},\]
where $f_*=\Nm_{K(Y)/K(X)}: K^M_d(K(Y)) \to K^M_d(K(X))$ is the norm map.
\end{prop}
\begin{proof}
First note that the sum in \ref{HS5} is finite. Indeed given points  $x_i\in X_{(i)}$ and $y_i\in Y_{(i)}$ with $f(y_i)=x_i$, then $f$ induces
a projective and surjective $K$-morphism $f_i: \ol{\{y_i\}}\to \ol{\{x_i\}}$. As source and target of $f_i$ have the same dimension it follows that 
for any $x_{i-1}\in \ol{\{x_i\}}^{(1)}$ the preimage $f_i^{-1}(x_{i-1})$ consists of $1$-codimensional points in $\ol{\{y_i\}}$, in particular it is a 
discrete noetherian  topological space and hence is finite. Thus there are only finitely many maximal chains in $Y$ lying over $\ul{x}$.

To prove the equality in \ref{HS5} we proceed by induction on the dimension $d$.
Set $E:=K(X)$ and $L:=K(Y)$. For $d=0$, \ref{HS5} translates by \ref{HS2} into
\[\beta\cdot\Tr_{L/K}(a)=[L:E]\cdot \beta \cdot \Tr_{E/K}(a) \quad\text{for } a\in F(E), \beta\in \Z,\]
which  holds since $[L:E]\cdot a= \Tr_{L/E}(a)$. 
Now assume $d\ge 1$ and the formula holds in dimension $\le d-1$. 
Write $\ux=(x_0,\ldots,x_d)$. We consider two cases.

{\em 1st case: $a\in F(\sO_{X,x_{d-1}})$.}
Set  $u:=x_{d-1}\in X^{(1)}$ and denote by $X'=\ol{\{u\}}\subset X$ the closure.
Set $\ux'=(x_0,\ldots, x_{d-1})\in\mc(X')$.
We first collect some standard commutative diagrams for Milnor $K$-theory, also to clarify the notation used later;
\eq{prop:HS51}{\xymatrix{
    &
K^M_{d}(K^h_{X,\ux})\ar[r]^{\partial_{\ux}}&
K^M_{d-1}(K^h_{X',\ux'})\\
K^M_d(E)\ar[r]^-{\iota_{u}}\ar@<1ex>[ur]^{\iota_{\ux}} &
K^M_d(K^h_{X, u})\ar[r]^-{\partial_{u}}\ar[u]_{\iota^{u}_{\ux}}&
K^M_{d-1}(K(X')),\ar[u]_{\iota_{\ux'}}
}}
where $\iota$'s are the natural maps and $\partial$'s are induced by the tame symbol, see \ref{para:hltame};
\eq{prop:HS52}{\xymatrix{
     K^M_d(L)\ar[d]_{\Nm_{L/E}}\ar[r]^-{\prod \iota_z} &
     \prod_{\substack{ z\in Y_{d-1} \\f(z)=u}} K^M_d(K^h_{Y, z})\ar[d]^{\sum \Nm_{z/u}}\\
     K^M_d(E)\ar[r]^-{\iota_{u}} &
     K^M_d(K^h_{X, u});
}}
\eq{prop:HS53}{\xymatrix{
K^M_d(K^h_{Y, z})\ar[r]^{\partial_{z}}\ar[d]_{ \Nm_{z/u} }&
K^M_{d-1}(K(z))\ar[d]^{ \Nm_{K(z)/K(u)} }\\
K^M_d(K^h_{X, u})\ar[r]^{\partial_{u}} &
K^M_{d-1}(K(u)),
}}
where $z\in Y_{d-1}$ with $f(z)=u$.
The commutativity of the above diagrams follows from  standard relations in Milnor $K$-theory,
e.g., in \cite[1.]{Rost} see  {\bf R3a} for \eqref{prop:HS51}, {\bf R1c} for \eqref{prop:HS52}, and  {\bf R3b} for \eqref{prop:HS53}.
With this notation we want to show
\[\sum_{\substack{\uy\in \mc(Y)\\ f(\uy)=\ux}}(f^*a, \iota_{\uy}\beta)_{Y/K,\uy}=(a, \iota_{\ux}\Nm_{L/E}\beta)_{X/K,\ux}\]
assuming $a\in F(\sO_{X,u})$ and that the equality holds in smaller dimensions.
We compute
\begin{align*}
(a, \iota_{\ux}\Nm_{L/E}\beta)_{X/K,\ux} & = \big(a(u), \partial_{\ux}(\iota_{\ux}\Nm_{L/E}\beta)\big)_{X'/K,\ux'} & &\text{by \ref{HS2}}\\
                                                             & = \sum_{\substack{z\in Y_{d-1} \\ f(z)=u}} \big(a(u), \iota_{\ux'}\Nm_{K(z)/K(u)}(\partial_z\iota_z \beta)\big)_{X'/K, \ux'}& &
                                                             \text{by (\ref{prop:HS5}.$*$)}\\
                                                             &= \sum_{\substack{z\in Y_{d-1}\\ f(z)=u}} \sum_{\substack{\uy'\in \mc(Y(z))\\ f(\uy')=\ux'}}
                                                                  \big(f_{Y(z)}^*a(u), \iota_{\uy'}\partial_z(\iota_z\beta)\big)_{Y(z)/K,\uy'} & &
                                                                  \parbox[c]{3cm}{\begin{footnotesize} by induction\\ $Y(z)=\ol{\{z\}}$\end{footnotesize}}\\
                                                             &= \sum_{\substack{z\in Y_{d-1} \\ f(z)=u}} \sum_{\substack{\uy'\in \mc(Y(z))\\ f(\uy')=\ux'}}
                                                                  \big(f_{Y(z)}^*a(u), \partial_{(\uy',\xi)}(\iota_{(\uy',\xi)}\beta)\big)_{Y(z)/K,\uy'} & & \text{by \eqref{prop:HS51}}\\
                                                            &= \sum_{\substack{\uy\in \mc(Y)\\ f(\uy)=\ux}} (f^*a, \iota_{\uy}\beta)_{Y/K,\uy} & &\text{by \ref{HS2}.}
\end{align*}
where $\xi$ is the generic point of $Y$.
This completes the proof of the first case.

{\em 2nd case: $a\in F(E)$.} By \ref{HS1} we may assume $X$ to be projective.
Let $D\subset X$ be a closed subscheme such that $X\setminus D$ is regular and $a\in F_{\gen}(X, D)$.
Enlarging $D$ we may assume $x_{d-1}\in D$,  $Y\setminus D_Y$ is regular, and $f^*a\in F_{\gen}(Y, D_Y)$, where $D_Y=Y\times_X D$.
By Corollary \ref{cor:approx} we find an element $\gamma\in K^M_d(L)$
such that 
\begin{enumerate}[label=(\roman*)]
\item\label{HS5i} $\gamma-\beta\in (\Vb_{d,Y|D_Y})_y$,  for all  $y\in Y^{(1)}\cap D_Y$ lying over $x_{d-1}$,
\item\label{HS5ii} $\gamma \in (\Vb_{d,Y|D_Y})_y$, for all $y\in Y^{(1)}\cap D_Y$ not lying over $x_{d-1}$,
\item\label{HS5iii}$\Nm_{L/E}(\gamma-\beta)\in (\Vb_{d,X|D})_{x_{d-1}}$,
\item\label{HS5iv}$\Nm_{L/E}(\gamma) \in (\Vb_{d,X|D})_z$, for all $z\in X^{(1)}\cap D\setminus\{x_{d-1}\}$.
\end{enumerate}
Set $\ux'':=(x_0,\ldots, x_{d-2}, x_d)\in \mc_{d-1}(X)$. We compute 
\begin{align*}
(a, \Nm_{L/E}\beta)_{X/K,x}& = \sum_{z\in \b(\ux'')\cap D}(a,\Nm_{L/E}\gamma)_{X/K, \ux''(z)}& &
\text{by \ref{HS5iii}, \ref{HS5iv}, \ref{HS3}} \\
                & =- \sum_{z\in \b(\ux'')\setminus D} (a,\Nm_{L/E}\gamma)_{X/K, \ux''(z)} & &
                \text{by \ref{HS4}} \\
 & =  -\sum_{z\in \b(\ux'')\setminus D} \sum_{\substack{\uy\in \mc(Y)\\f(\uy)=\ux''(z)}} (f^*a, \gamma)_{Y/K,\uy}& &
 \text{by the 1st case}\\
&= \sum_{z\in \b(\ux'')\cap D}\sum_{\substack{\uy\in \mc(Y)\\f(\uy)=\ux''(z)}} (f^*a, \gamma)_{Y/K,\uy} & &
\text{by \ref{HS4}}\\
&=\sum_{\substack{y\in\mc(Y)\\f(\uy)=\ux}} (f^*a, \beta)_{Y/K, \uy} & & \text{by \ref{HS5i},\ref{HS5ii}, \ref{HS3}.}
\end{align*}
Note that we can apply \ref{HS4} also in the case $d=1$, since $X$ and $Y$ are projective.
This completes the proof of the proposition.
\end{proof}
 
The following corollary will be used in the proof of Proposition \ref{prop:diag} and in  \cite{RS-AS}.
\begin{cor}\label{cor:HS5}
Let $f:Y\to X$ be a dominant and quasi-projective $K$-morphism between integral $K$-schemes 
of the same dimension $d$. Let $\ux=(x_0,\ldots, x_d)\in\mc(X)$ and $u:=x_{d-1}$.
Let $y\in Y^{(1)}$ with $f(y)=u$.
We assume that $f$ induces a projective  morphism 
between the closures of the points $y$ and $u$. 
Then $K^h_{Y, y}$ is finite over $K^h_{X, u}$ (see \ref{para:hlr} for notation)
and  for all $a\in F(K(X))$ and $\beta\in K^M_d(K_{Y,y}^h)$, 
we have
\[\tag*{(HS5$'$)}\label{HS5'}
\sum_{\substack{\uz \in \mc_{d-1}(Y),\, \uz<y\\ f(\uz(y))=\ux}} (f^*a, \beta)_{Y/K,\uz(y)}
=\big(a, \Nm_{y/u}(\beta)\big)_{X/K,\ux},\]
where $\uz<y$ means $z_{d-2}<y$ with $\uz=(z_0,\ldots, z_{d-2}, z_d)$ and
$\uz(y)=(z_0,\ldots, z_{d-2}, y,z_d)$,
and $\Nm_{y/u}: K^M_d(K^h_{Y, y})\to K^M_d(K^h_{X, u})$ is the norm map. 
\end{cor}
\begin{proof}
Set $E:=K(X)$ and $L:=K(Y)$. 
Note that $y$ is an isolated point in $f^{-1}(u)$, 
hence $\sO^h_{X,u}\to \sO^h_{Y,y}$ is finite and injective.
Let $Y\inj \ol{Y}\xr{\bar{f}} X$ be a projective compactification of $f$.
Take a closed subscheme $D\subset X$ such that $X\setminus D$ and $\ol{Y}\setminus D_{\ol{Y}}$ are regular,
where  $D_{\ol{Y}}=\ol{Y}\times_X D$, and $a\in F_{\gen}(X, D)$, $\bar{f}^*a\in F_{\gen}(\ol{Y},D_{\ol{Y}})$.
Set $X'=\Spec \sOh_{X,u}$ and denote by 
$\bar{f}': \ol{Y}'\to X'$ the base change of $\bar{f}$. Note that the total fraction ring of $\ol{Y}'$
is equal to $\bigoplus_{z\in \bar{f}^{-1}(u) } K^h_{\ol{Y},z}$.
By Corollary \ref{cor:approx} applied to $\bar{f}'$  and 
$(\beta_z)\in \bigoplus_{z\in \bar{f}^{-1}(u) } K^M_d(K^h_{\ol{Y},z})$ with $\beta_y=\beta$ and $\beta_z=0$ for $z\neq y$,
we find an element $\gamma'\in \bigoplus_{z\in \bar{f}^{-1}(u)}K^M_d(K^h_{\ol{Y},z})$ such that 
\begin{enumerate}[label=(\roman*)]
\item\label{cor:HS51} $\beta-\gamma'\in (\Vb_{d, \ol{Y}| D_{\ol{Y}}})^h_{y}$
\item\label{cor:HS52} $\gamma' \in (\Vb_{d, \ol{Y}| D_{\ol{Y}}})^h_{z}$,  for all 
$z\in \bar{f}^{-1}(u)\setminus \{y\}$,
\item\label{cor:HS53} $\Nm_{y/u}(\beta)-\Nm_{y/u}(\gamma')\in (\Vb_{d, X| D})^h_{u}$,
\item\label{cor:HS54} $\Nm_{z/u}(\gamma')\in (\Vb_{d, X|D})^h_u$,  for all 
$z\in \bar{f}^{-1}(u)\setminus \{y\}$.
\end{enumerate}
We have a surjection
\[ K^M_d(L) \surj \bigoplus_{z\in \bar{f}^{-1}(u)}\frac{K^M_d(K^h_{\ol{Y},z})}{(\Vb_{d, \ol{Y}|D_{\ol{Y}}})^h_z}.\]
Indeed, by Proposition \ref{para:Vbar2} it suffices to show this for  $Y$ normal in which case it follows
from the Approximation Lemma. Thus we find $\gamma\in K^M_d(L)$, so that 
\ref{cor:HS51}-\ref{cor:HS54} holds with $\gamma'$ replaced by $\gamma$.
Set $\ux'=(x_0,\ldots, x_{d-1})$. We compute
\begin{align*}
(a, \Nm_{y/u}(\beta))_{X/K, \ux} &= \sum_{\substack{z\in Y\\ \ol{f}(z)=u}} (a, \Nm_{z/u}(\gamma))_{X/K,\ux}  & &
\text{by \ref{cor:HS53}, \ref{cor:HS54}, \ref{HS3}}\\
&= (a, \Nm_{L/E}(\gamma))_{X/K,\ux}& &  \text{by \eqref{prop:HS52}}\\
 &= \sum_{\substack{\uz\in\mc(\ol{Y})\\ \ol{f}(\uz)=\ux}} (\bar{f}^*a, \gamma)_{\ol{Y}/K,\uz} &  &
\text{by \ref{HS5}}\\
&= \sum_{\substack{\uz\in\mc_{d-1}(\ol{Y}),\, \uz<y\\ \ol{f}(\uz(y))=\ux}} (\bar{f}^*a, \beta)_{\ol{Y}/K,\uz(y)} & &
\text{by \ref{cor:HS51}, \ref{cor:HS52}, \ref{HS3}}\\
&= \sum_{\substack{\uz\in\mc_{d-1}(Y), \, \uz<y\\ f(\uz(y))=\ux}} (\bar{f}^*a, \beta)_{\ol{Y}/K,\uz(y)}, 
\end{align*}
where  the last equality follows from the assumption that the closure of $y$ in 
$\ol{Y}\times_X \ol{\{u\}}$ is already closed in $Y\times_X \ol{\{u\}}$.
This completes the proof of the corollary.
\end{proof}

\begin{exs-rmks}\label{exs-rmks:res}
We fix a function field $K$ over $k$ and $X$ is an integral  scheme of finite type over $K$ and  dimension $d$ and $\ux\in \mc(X)$.
\begin{enumerate}[label=(\arabic*)]
\item\label{exs-rmks:res1} Let $F=\Omega^j_{-/k}$. In \cite[p. 515]{Lo} a residue homomorphism
\[\Res_{X/K,\ux}: \Omega^{j+d}_{K(X)/k}\to \Omega^{j}_{K/k}\]
is defined, which generalizes a construction of Par\v{s}in for surfaces, see \cite[\S 1]{Parsin}. 
(In \cite{Lo} it is denoted by $\Res^f_\ux$, where $f:X\to \Spec K$ is the structure homomorphism.)
We have for all $a\in F(K(X))$ and $\beta\in K^M_d(K(X))$
\eq{exc-rmks:res1.1}{(a,\beta)_{X/K,\ux}=\pm \Res_{X/K,\ux}(a\cdot\dlog \beta),}
where $\pm$ is a universal sign depending on the choice of the sign for the tame symbol and for $\Res$.
Indeed, by the construction in {\em loc. cit.} the right hand side satisfies \ref{HS1} and \ref{HS3}, 
\ref{HS2} (up to sign) holds by Lemma 12 and \ref{HS4} by Theorem 3. 
Thus \eqref{exc-rmks:res1.1} follows from the uniqueness statement in Proposition \ref{prop:HLS}.
Note that the construction of $\Res_{X/K,\ux}$ in \cite{Parsin} and \cite{Lo} are completely different in spirit:
the residue is defined in an explicit way using power series in several variables, property \ref{HS3} holds by construction, 
whereas \ref{HS4} is a theorem. On the other hand for the symbol $(-,-)_{X/K,\ux}$ the
reciprocity law \ref{HS4} follows immediately from the definition whereas \ref{HS3} is implied by Theorem \ref{thm:mod-pairing}, a main result of \cite{RS-ZNP}.
\item It follows from the uniqueness statement that in case $d=1$ the symbol $(-,-)_{X/K,\ux}$ 
agrees with the local symbol for reciprocity sheaves considered in \cite[Proposition 5.2.1]{KSY1}, which
generalizes the local symbol for commutative $k$-groups by Rosenlicht-Serre, see \cite{Serre-GACC}.
\item Let $F=H^1((-)_{\et}, \Q/\Z)$. 
Assume $K=k$ is a finite field. Then \eqref{defn:HLS1} induces a morphism
\[K^M_{d}(K^h_{X,\ux})\to \Hom(F(K^h_{X,\ux}), F(k))=\pi_1^{\rm ab}(K^h_{X,\ux}),\]
where the equality on the right follows from $F(k)=\Q/\Z$. 
By \cite[Proposition 3.3]{KS-GCFT}
this map  decomposes as the product of the maps
\[ K^M_{d}(\Frac(V))\to   \pi_1^{\rm ab}(\Frac(V)),\] 
where $V$ runs through all $d$-discrete valuation rings of $k(X)$ dominating $\ux$ 
(in the sense of \cite[(3.2)]{KS-GCFT}). 
It can be shown that the induced morphisms $K^M_{d}(\Frac(V))\to  \pi_1^{\rm ab}(\Frac(V))$ coincide with
Kato's  higher local reciprocity maps $\Psi_{\Frac(V)}$ constructed in \cite[p. 661]{Kato-LCFTII}
(see \cite[Theorem 3.5]{KS-GCFT} for an explanation of how to deduce the henselian case from the complete case
treated in \cite{Kato-LCFTII}). We leave the verification of this fact to interested readers.
\item For $F=W_n\Omega^j$, the $j$-th de Rham-Witt differentials (see \cite{Il}), we get higher local  symbols
\[(-,-)_{X/K,\ux}: W_n\Omega^j_{K(X)}\otimes_{\Z} K^M_{d}( K^h_{X,\ux})\to W_n\Omega^j_K,\]
which by the above generalizes the residue symbol from the one-dimensional case constructed in 
\cite[\S 2]{Kato-LCFTII} or \cite[2.]{R} (see also \cite{R-Err}).
Since Verschiebung, Frobenius, restriction, and the differential on the de Rham-Witt complex are morphisms of
reciprocity sheaves, it follows that the above symbol is compatible with these.
Note also  that from \cite[\S 2.5, Lemma 12]{Kato-LCFTII} and \cite[Proposition 3.3]{KS-GCFT},
one can define a residue morphism (similar as in \cite{Lo})
\[\Res_{X/K,\ux}: W_n\Omega^{j+d}_{K(X)}\to W_n\Omega^j_{K}.\]
The connection to the residue symbol defined here should be as in \eqref{exc-rmks:res1.1}.
One way to verify this would be to show that $\Res_{X/K,\ux}$ satisfies \ref{HS4}, e.g., using the strategy from \cite{Lo},
 and then use uniqueness (the properties \ref{HS1}-\ref{HS3} are direct to check from the construction).
 We leave the details to interested readers.
\item Let $F=G$ be a commutative $k$-group. We obtain a pairing
\[(-,-)_{X/K,\ux}: G(K(X))\otimes_\Z K^M_d( K^h_{X,\ux} )\to G(K).\]
A pairing like this was also defined in \cite[III]{KS-2CFT}.
(In {\em loc. cit.} such a pairing was defined for higher dimensional local fields, but one can use
\cite[Proposition 3.3]{KS-GCFT} to obtain an induced pairing as above.)
In characteristic zero, one can check that the two pairings coincide. Indeed, in this case
it suffices to consider $G=\G_a$, $\G_m$, or an abelian variety. For the $\G_a$-case they are induced by 
$\Res_{X/K,\ux}$ from \ref{exs-rmks:res1} by \cite[p. 144]{KS-2CFT}.
If $G=\G_m$, it is induced for both pairings by an iteration of the tame symbol.
For an abelian variety, \ref{HS2} and the formula in the middle of p. 145 
in \cite{KS-2CFT} show that they coincide. In particular we find that the symbol from {\em loc. cit.} in characteristic zero satisfies the reciprocity theorem \ref{HS4}.
In positive characteristic  we believe that the two pairings coincide, but this remains to be checked.
\item Let $(Y,E)$ be a proper modulus pair (see \ref{para:RSC}).
The presheaf with transfers  $h_0(Y,E)$ is defined in \cite[Definition 2.2.1]{KSY2} and it is a 
reciprocity presheaf by \cite[Theorem 2.3.3]{KSY2}. Thus the Nisnevich sheafification
$F=h_0(Y,E)_\Nis$ is a reciprocity sheaf by \cite[Theorem 0.1]{S-purity}.
By \cite[Theorem 3.2.1]{KSY2}, we have $h_0(Y,E)(K)=\CH_0(Y_K|E_K)$, the Chow group  of zero-cycles with modulus introduced in \cite{Kerz-Saito}.
Therefore \eqref{defn:HLS1} induces a morphism
\[(-,-)_{X/K,\ux}:\CH_0(Y_{K(X)}|E_{K(X)})\otimes_\Z K^M_d(K^h_{X,\ux}) 
\to  \CH_0(Y_K|E_K),\]
satisfying \ref{HS1}-\ref{HS5}.  
If $\alpha\in \CH_0(Y_{K(X)}, E_{K(X)})$
can be represented by a zero-cycle in $Y_{K(X)}\setminus E_{K(X)}$, which spreads out  to a finite correspondence
$\tilde{\alpha}$ from $U$ to $Y\setminus E$ for some smooth open $U\subset X$ {\em containing the closed point $x_0$ of $\ux$},
 then it follows from \ref{HS2}, that we have 
\[(\alpha, \beta)_{X/K,\ux}= \partial_\ux(\beta)\cdot f_*{\tilde{\alpha}}\;\;\text{for $\beta\in K^M_d(K^h_{X,\ux})$},\]
where $\partial_\ux: K^M_d(K^h_{X,\ux})\to K^M_0(\kappa(x_0))=\Z$ is the map \eqref{para:hltame1} and $f: U\times (Y\setminus E) \to U$ is the projection.
\end{enumerate}
\end{exs-rmks}

\section{Characterization of the modulus via higher local symbols}\label{sec:mod-LS}
In this section $k$ is a perfect field and $F\in\RSC_\Nis$.
The main result of this section is the following.

\begin{thm}\label{thm:LS}
Let $(X, D)$ be a modulus pair  (see \ref{para:RSC}) 
with $X$ of pure dimension $d$. Let $U=X\setminus|D|$ and $a\in F(U)$.
For a function field $K/k$ denote by $X_K=X\otimes_k K$ the base change and by $a_K\in F(U_K)$ the pullback of $a$.
Let $W\subset |D|$  be a set of closed points which contains at least one  point of every  irreducible component of $|D|$.
Consider the following conditions (see Definition \ref{defn:HLS}):
\begin{enumerate}[label=(\roman*)]
\item\label{thm:LS1} $a\in \tF(X,D)$.
\item\label{thm:LS2} 
For any function field $K$ and $\ux=(x_0,\ldots, x_{d})\in \mc(X_K)$ we have 
\[(a_K,\beta)_{X_K/K,\ux}=0,\quad \text{for all } \beta\in (V_{d, X_K|D_K})^h_{(x_0,\ldots, x_{d-1})}.\]
\item\label{thm:LS3} For any function field $K$ and $\ux=(x_0,x_1,\ldots, x_d)\in \mc(X_K)$ with $x_0\in W_K$ and $x_{d-1}\in D_K$, we have
\[(a_K,\beta)_{X_K/K,\ux}=0,\quad \text{for all } \beta\in (V_{d, X_K|D_K})_{x_{d-1}}.\]
\end{enumerate}
Then, we have the implication 
$\ref{thm:LS1}\Longrightarrow \ref{thm:LS2}\Longrightarrow \ref{thm:LS3}$.
Assume furthermore that there exists an open dense immersion $X\inj \Xb$ into a smooth and projective $k$-scheme, such that 
$\Xb\setminus U$ is an SNCD, then all the above statements are equivalent.
\end{thm}
We stress the fact that $(V_{d, X|D})^h_{(x_0,\ldots, x_{d-1})}$ in \ref{thm:LS2} is the limit over all Nisnevich neighborhoods of $(x_0,\ldots, x_{d-1})$
(see \ref{para:hlr}), whereas $(V_{d,X|D})_{x_{d-1}}$ in \ref{thm:LS3} is the Zariski stalk at the one codimensional point  $x_{d-1}\in X^{(1)}$. In section \ref{sec:diag} we will see that in case $D$ is reduced, the assumption on the 
existence of a smooth compactification is superfluous (but still $X$ has to be smooth and $D$ is a SNCD).


\begin{para}\label{para:Hdvan}
Before we prove  Theorem \ref{thm:LS} we recall the following result:

\medskip
\noindent {\em Let $X$ be a separated scheme of finite type over a field $K$ of dimension $d$.
Assume no irreducible component of dimension $d$ of $X$ is  {\em proper}.
Then for any coherent sheaf $\sF$ on $X$, we have }
\[H^d(X_\Zar,\sF)=0.\]
\medskip

This theorem was conjectured by Lichtenbaum and proven by Grothendieck, 
see \cite[Theorem 6.9]{HaLC} for Grothendieck's proof in the quasi-projective case relying on duality theory, see \cite{Kleiman}
for a more elementary proof in the stated generality.
We will use the following consequence (cf. the proof of \cite[Theorem 6.9]{HaLC}):

Let $Y$ be a proper $K$-scheme of dimension $d$. Let $W\subset Y$ be a set of closed points which contains at least one 
closed point of each irreducible component of dimension $d$ of $Y$.
Then the natural map
\eq{para:Hdvan1}{\bigoplus_{z\in W }H^{d}_z(Y_{\Zar},\sF)\to H^{d}(Y_{\Zar},\sF)}
is surjective for all coherent $\sO_{Y}$-modules $\sF$.
Indeed, this holds as $H^d((Y\setminus W)_\Zar,\sF)$ vanishes by the above result.
\end{para}

\begin{proof}[Proof of Theorem \ref{thm:LS}.]
The implication $\ref{thm:LS1}\Rightarrow\ref{thm:LS2}$ follows from \ref{HS3} in Proposition \ref{prop:HLS} and  $\tF(X,D)\subset F_{\gen}(X,D)$,  see \eqref{para:mod-pairing0}.
  The implication $\ref{thm:LS2}\Rightarrow\ref{thm:LS3}$ is immediate from the natural map 
  $(V_{d, X_K|D_K})_{x_{d-1}}\to (V_{d, X_K|D_K})_{(x_0,\ldots, x_{d-1})}^{h}$.
Assume $X$ has  a smooth compactification $\ol{X}$ such that $\ol{X}\setminus U$ is an SNCD.
It remains to show that in this case also $\ref{thm:LS3}\Rightarrow \ref{thm:LS1}$ holds.
Let $\Db\subset\Xb_K$ be the closure of $D_K$. 
By Theorem \ref{thm:mod-pairing}\ref{thm:mod-pairing2} and Lemma \ref{lem:HLS}\ref{lem:HLS2} it suffices  to prove the following:

\begin{claim}\label{claim;thm;HLS0}
Let $B\subset \Xb_K$ be an effective Cartier divisor supported on $\Xb_K\setminus X_K$ and $n\ge 1$.
Set $Y=|\Db+B|$ which we will view as a reduced effective Cartier divisor.
Then the sequence
\begin{multline*}
 \bigoplus_{w\in W_K} \bigoplus_{\ux_w=(w,x_1,\ldots, x_d)} \; (V_{d,X_K|D_K})_{x_{d-1}} 
\to H^d(\Xb_{K,\Nis},V_{d, \ol{X}_K|\Db+nY+B}) \\
\to H^d(\Xb_{K,\Nis}, V_{d, \Xb_K|\Db+B}) \to 0\end{multline*}
is exact. Here, the second sum is over all
$\ux_w=(w,x_1,\ldots,x_{d-1},x_d)\in \mc(X_K)$ such that $x_{d-1}\in D^{(0)}$ and the first map is the sum of the maps
\[ (V_{d, X_K|D_K})_{x_{d-1}}\to (V_{d,X_K|D_K})^h_{\ux_w'} \to
(V_{d,X_K|D_K})^h_{\ux_w}
\to H^d(\Xb_{K,\Nis}, V_{d, \Xb_{K}|\Db+nY+B}) \]
where $\ux_w'=(w,x_1,\dots,x_{d-1})$ for $\ux_w=(w,x_1,\ldots, x_d)$ and
the last map is induced by $c_{\ux_w}$ from \eqref{para:hlc2} noting
$(V_{d,X_K|D_K})^h_{\ux_w}=(V_{d, \Xb_{K}|\Db+nY+B})^h_{\ux_w}=K^h_{X_K,\ux_w}$.
\end{claim}

Considering the same claim with $\Db+B$ on the right replaced by $\Db+mY+B$ and $\Db+nY+B$ 
in the middle replaced by $\Db+(m+1)Y+B$,
for $m=0,\dots,n-1$, we are reduced to the case $n=1$.
Let 
\[ \sG=V_{d, \Xb_K|\Db+B}/V_{d,\Xb_K|\Db+Y+B}.\]
It is supported on $Y$, and by \cite[Corollary 2.12]{RS-CycleMap} we have an exact sequence
\eq{eq1;claim;thm;HLS0}{
H^{d-1}(Y_{\Zar},\sG) \to H^d(\Xb_{\Nis},V_{d,\Xb|\Db+Y+B})\to  H^d(\Xb_{\Nis},V_{d,\Xb|\Db+B})\to 0.}
Then, we obtain surjections
\begin{multline*}
\bigoplus_{w\in W_K} \bigoplus_{\ux_w=(w,x_1\ldots, x_d)} (V_{d, X_K|D_K})_{x_{d-1}}\surj \bigoplus_{w\in W_K} \bigoplus_{\ux_w'=(w,x_1\ldots, x_{d-1})} \sG_{x_{d-1}} \\
\xrightarrowdbl{c^{\Zar}_{\ux_w',0}} \bigoplus_{w\in W} H^{d-1}_{w}(Y_{\Zar},\sG)\surj H^{d-1}(Y_{\Zar},\sG).\end{multline*}
where the last map is surjective due to \eqref{para:Hdvan1} and $c^{\Zar}_{\ux_w',0}$ is the map \eqref{para:hlc3Zar}
and it is surjective by Remark \ref{prop:mapNisCoh;rem1}. 
This proves Claim \ref{claim;thm;HLS0} and hence the theorem.
\end{proof}

\begin{rmk}
If  $F=W_n$  is the sheaf  of $p$-typical Witt vectors of length $n$, where $p={\rm char}(k)$,
the equivalence $\ref{thm:LS1}\Leftrightarrow \ref{thm:LS3}$ in Theorem \ref{thm:LS} is reminiscent
of \cite[Proposition 7.5]{Kato-Russell} (the case ${^\flat {\rm fil}^F_m}$). Though in {\em loc. cit.}
$k$ is assumed to be algebraically closed and  it is not necessary to consider all function fields $K/k$.
\end{rmk}

\begin{defn}\label{def;FLS}
For $(X,D)\in \ulMCor$ with $U=X-|D|$, We define
\[F^{\rm LS}(X,D):=\left\{a\in F(U)\,\Bigg\vert\,
\begin{minipage}{9.5cm}
$(a_K,\beta)_{X_K/K,\ux}=0$ \text{ for } $\forall K/k$, $\forall \ux=(x_0,\ldots, x_d)\in\mc(X_K)$, $\forall\beta\in (V_{d, X_K|D_K })_{x_{d-1}}$
\end{minipage}
 \right\}.\]
By Theorem \ref{thm:LS} we always have an inclusion $\tF(X,D)\subset F^{\rm LS}(X,D)$ and this is an 
equlaity if $X$ has a smooth projective compactification $\Xb$ such that $\Xb\setminus U$ is SNCD.
\end{defn}

\begin{para}\label{para:level}
We say that a reciprocity sheaf $F$ has {\em level $n\ge 0$}, 
if for any smooth $k$-scheme $X$ and any $a\in F(\A^1\times X)$ the following implication holds:
\[a_{\A^1_z}\in F(z)\subset F(\A^1_z), \quad \text{for all } z\in X_{(\le n-1)}  
\Longrightarrow a\in F(X)\subset F(\A^1\times X),\]
where $a_{\A^1_z}$ denotes the restriction of $a$ to $\A^1_z=\A^1\times z\subset \A^1\times X$,
$X_{(\le n-1)}$ denotes the set of points in $X$ whose closure has dimension $\le n-1$, 
and for a smooth scheme $S$ we identify $F(S)$ with its image in $F(\A^1\times S)$ via pullback along the projection map.
This is equivalent to  the motivic conductor of $F$ having level $n$ in the language of \cite{RS}.
The $\A^1$-invariant sheaves with transfers are precisely the reciprocity sheaves of level $0$. 
By \cite[Part 2]{RS}, the presheaves $X\mapsto G(X)$, with $G$ a commutative algebraic group,
$X\mapsto \Hom(\pi_1^{\rm ab}(X), \Q/\Z)$,  
$X\mapsto {\rm Lisse}^1(X)$, the lisse $\ol{\Q_{\ell}}$-sheaves of rank $1$,
and  $X\mapsto {\rm Conn}^1_{\rm int}(X)$, the integrable rank 1 connections on $X$ (${\rm char}(k)=0$), 
are reciprocity sheaves of level $1$;
and  the presheaves $X\mapsto \Omega^1(X)$, $X\mapsto Z\Omega^2(X)$ (both in ${\rm char}(k)=0$), 
and $X\mapsto H^1(X_{\rm fppf}, G)$, with $G$ a finite flat $k$-group scheme, are reciprocity sheaves of level $2$.
\end{para}

We say that \emph{resolutions of singularities hold over $k$ in dimension $\le n$},
if for any integral projective $k$-scheme
$Z$ of dimension $\le n$ and any effective Cartier divisor $E$ on $Z$, there exists a proper birational morphism $h:Z'\to Z$
such that $Z'$ is regular and $|h^{-1}(E)|$ has simple normal crossings. 
This is known to hold if ${\rm char}(k)=0$ by Hironaka or if $n\leq 3$ by \cite{Cossart-Piltant}.

\begin{cor}\label{cor:HLS}
Assume $F$ has  level $n\ge 0$ and 
resolutions of singularities hold over $k$ in dimension $\le n$.
Let $(X,D)$ ba a modulus pair. Assume  $X$ is quasi-projective and  set $U=X\setminus |D|$. Let $a\in F(U)$.
The following statements are equivalent
\begin{enumerate}[label=(\roman*)]
\item\label{cor:HLS1} $a\in\tF(X,D)$;
\item\label{cor:HLS2} $h^*a\in F^{\rm LS}(Z, h^*D)$, for all $k$-morphisms 
$h:Z\to X$ with $Z$ smooth, quasi-projective with  $\dim(Z)\le n$ such that  $|h^*D|$ is SNCD.
\end{enumerate}
\end{cor}
\begin{proof}
By \ref{HS3} we only have to show the implication $\ref{cor:HLS2}\Rightarrow\ref{cor:HLS1}$. 
Let $h:Z\to X$ be as in \ref{cor:HLS2}. By resolution of singularity in dimension $\le n$, 
Theorem \ref{thm:LS} together with Theorem \ref{thm:mod-pairing}\ref{thm:mod-pairing2}
imply $F^{\rm LS}(Z, h^*D)=F_{\gen}(Z, h^*Z)=\tF(Z, h^*Z)$. Hence $a\in \tF(X,D)$, by \cite[Corollary 6.10]{RS-ZNP}
\end{proof}

\begin{rmk}
Note  that if ${\rm char}(k)=0$ or ${\rm char}(k)=p>0$ and $F$ has level 
$\le 3$ (see the examples listed in \eqref{para:level} and also \cite[Example 6.11(5)]{RS-ZNP}), then Corollary \ref{cor:HLS} yields unconditional results.
\end{rmk}

\section{A diagonal argument}\label{sec:diag}

We give a refinement of Theorem \ref{thm:LS} in the case where $D$ is reduced, see Proposition \ref{prop:diag} below.
In this section $k$ is a perfect field and $F\in\RSC_\Nis$.

\begin{lem}\label{lem:diag}
Let $S$ be a smooth integral $k$-scheme of dimension $d-1$ and $X=\A^1_S=S[z]$.
Let $L=k(X)$ be the function field  and set $X_L=X\otimes_k L$. Write
\[t=z\otimes 1, \,\, s=1\otimes z\in \sO_{X_L}.\]
Let $\eta\in (S\times_k S)^{(d-1)}$ be the generic point of the diagonal and let
$\theta_1,\ldots, \theta_{d-1}$ be a regular sequence of parameters in $A:=\sO_{S\times_k S ,\eta}$.
Then $A[t,s]/(t, \theta_1,\ldots, \theta_{d-1-i})$ is integral for $i=0,\ldots, d-1$.
We denote by $x_i\in X_L\subset X\times_k X$,  the generic point of the image of the natural map
\[\Spec A[t,s]/(t, \theta_1,\ldots, \theta_{d-1-i})\to X\times_k X,\]
and by $x_d\in X_L^{(0)}$ the generic point of the component containing $x_0$.
Set 
\[\beta_0:=\left\{\frac{t-s}{t-1}, \theta_1,\ldots, \theta_{d-1}\right\},\quad 
\gamma_0:=\left\{s, \theta_1,\ldots, \theta_{d-1}\right\}\in K^M_d(k(x_d)).\]

Then $\ux=(x_0,\ldots, x_d)\in \mc(X_L)$ and 
\eq{lem:diag1}{\pm (a_L,\beta_0)_{X_L/L,\ux}= j^*(a-\lambda^*a)\in F(L),\quad 
\text{for all }a\in F(\P^1\setminus 0_S, \infty_S),}
\eq{lem:diag2}{(a_L,\gamma_0)_{X_L/L,\ux}=0,\quad \text{for all }  a\in \tF(\P^1_S\setminus 0_S),}
where $a_L\in F(X_L[t^{-1}])$ 
denotes the restriction of $a$ to $X_L$, $j:\Spec L\to X$ is the inclusion of the generic point, and 
the map $\lambda$ is the composition
\[X\xr{\text{proj.}} S \cong V(z-1)\inj X.\]
\end{lem}
\begin{proof}
By, e.g., \cite[Theorems 14.2 and 3]{Matsumura} the ring $A[t,s]/(t, \theta_1,\ldots, \theta_{d-1-i})$ is integral, regular and of dimension $i+1$.
It follows that  $x_i\in X_{L}$ is a point of dimension $i$ so that $\ux\in \mc(X_L)$. 
We first show \eqref{lem:diag1}. Let $a\in \tF(\P^1_S\setminus 0_S,\infty_S)$.
Assume $d-1\ge 1$ and set $\ux'=(x_0, \ldots, x_{d-2}, x_d)\in \mc_{d-1}(X_L)$.
By \ref{HS4}
\[ (a_L,\beta_0)_{X_L/L,\ux}= -\sum_{\substack{y\in \b(\ux')\\ y\neq x_{d-1}}} (a_L,\beta_0)_{X_L/L,\ux'(y)}.\]
The maximal ideal in $\sO_{X_L,x_{d-2}}$ is generated by $t$, $\theta_1$  and thus the only point $y\in X_L^{(1)}$ with $y> x_{d-2}$,
at which $\beta_0$ is not regular  is given by 
$z_{d-1}= \Spec \sO_{X_L,x_{d-2}}/(\theta_1)$,
whereas $x_{d-1}= \Spec \sO_{X_L,x_{d-2}}/(t)$ is the only such point at which  $a_L$ is not regular.
Thus \ref{HS2} yields

\begin{align*}
 (a_L,\beta_0)_{X_L/L,\ux}=  - (a_L(z_{d-1}),\beta_1)_{Z_{d-1}/L, (x_0,\ldots,x_{d-2},z_{d-1})},
\end{align*}
where $Z_{d-1}=\ol{\{z_{d-1}\}}\subset X_L$, and 
\[\beta_1=\partial_{z_{d-1}} \beta_0 =\pm \left\{\frac{t-s}{t-1}, \theta_2,\ldots, \theta_{d-1}\right\}\in K^M_{d-1}(L(Z_{d-1})).\]
If $d-2=0$ we stop, if $d-2\ge 1$ we observe that $a_L(z_{d-1})\in \tF(Z_{d-1}[t^{-1}])$ and we proceed by applying 
\ref{HS4} and \ref{HS2} again. Iterating yields
\[(a_L,\beta)_{X_L/L,\ux}=\pm \left(a_L(z_1), \frac{t-s}{t-1}\right)_{Z_1/L, (x_0, z_1)},\]
where $z_1\in Z_1$ is the generic point of $\Spec \sO_{X_L, x_0}/(\theta_1,\ldots, \theta_{d-1})$ and $Z_1$ is its closure in $X_L$. 
By the choice of the $\theta_i$, we have $Z_1= \Spec L[t]\subset \P^1_L$ and hence we may identify $x_0$ with $0_L\in \P^1_L$.
Applying \ref{HS1} and \ref{HS4} one more time we find
\[(a_L,\beta_0)_{X_L/L,\ux}=\pm\sum_{y\in\P_L\setminus 0_L} \left(a_L(z_1), \frac{t-s}{t-1}\right)_{\P^1_L/L, (y, z_1)}.\]
Note that under the identification $Z_1= \Spec L[t]\subset \P^1_L$, the element $a_L(z_1)$ corresponds to the pullback of $a$ along $\P^1_{L}\to \P^1_S$ 
induced by the natural inclusions $\sO_{S}\subset k(S)\subset k(S)(s)=L$. Thus $a_L(z_1)\in \tF(\P^1_L\setminus 0_L, \infty_L)$ by the choice of $a$ in \eqref{lem:diag1}. 
Since  $\frac{t-s}{t-1}\in (V_{1, \P^1_L|\infty_L})_{\infty_L}$ we obtain by \ref{HS3} and \ref{HS2}
\[(a_L,\beta_0)_{X_L/L,\ux}=\pm \big(a_{L}(z_1)_{|\{t-s=0\}}- a_{L}(z_1)_{|\{t-1=0\}}\big)\in F(L).\]
This yields \eqref{lem:diag1}. Now assume $a\in F(\P^1_S\setminus 0_S)$. 
The same argument as above with $\beta_0$ replaced by $\gamma_0$ yields
\[(a_L,\gamma_0)_{X_L/L,\ux}=\pm\sum_{y\in\P_L\setminus 0_L} \left(a_L(z_1), s\right)_{\P^1_L/L, (y, z_1)}.\]
This vanishes by \ref{HS2} since $a_L(z_1)\in F(\P^1_L\setminus 0_L)$ and $s\in L^\times$.
Hence \eqref{lem:diag2}.
\end{proof}

\begin{cor}\label{cor:diag}
Let $S$ be a smooth integral $k$-scheme of dimension $d-1$ and $X=\A^1_S=S[z]$.
Let $L=k(X)$ be the function field  and set $X_L=X\otimes_k L$ 
and $S_L=S\otimes_k L$. 
Let $\iota: S_L=V(t)\inj X_L $  be the closed immersion defined by $t=0$, where $t=z\otimes 1$.
Denote by $s_0\in S_L$ the image of  the generic point of the diagonal in $S\times S$ under the map 
$S\times_k S\to S\times_k X\to S_L$ where the first map is the base change of 
the closed immersion $S\inj X$ defined by $z=0$.
Denote by $\eta\in X_L$ the generic point of the irreducible component containing $\iota(s_0)$.
Let $a\in \tF(\P^1_S\setminus 0_S, \infty_S)\subset F(X[z^{-1}])$.
\begin{enumerate}[label=(\arabic*)]
\item\label{cor:diag1} Assume for all $\uy=(y_0,\ldots, y_{d-1})\in \mc(S_L)$ with $y_0=s_0$ we have 
\[(a_L, \beta)_{X_L/L, (\iota(\uy), \eta)}= 0,\quad \text{for all }\beta\in K^M_{d}(\sO_{X_L, \iota(y_{d-1})}).\]
Then $a\in {\rm Im}(F(S)\to \tF(\P^1_S\setminus 0_S,\infty_S))$.
\item\label{cor:diag2} Assume for all $\uy=(y_0,\ldots, y_{d-1})\in \mc(S_L)$ with $y_0=s_0$ we have 
\[(a_L, \beta)_{X_L/L, (\iota(\uy), \eta)}= 0,\quad \text{for all }\beta\in (V_{d, X_L|0_{S_L}})_{\iota(y_{d-1})}.\]
Then $a\in \tF(\P^1_S, 0_S+\infty_S)$.
\end{enumerate}
\end{cor}
\begin{proof}
By \eqref{lem:diag1} and with the notation from there, the condition in \ref{cor:diag1}
implies $j^*a=j^*\lambda^*a$. Since $j^*: F(X)\to F(L)$ is injective, by  \cite[Theorem 3.1]{S-purity}, we obtain the 
first statement. We show the statement in \ref{cor:diag2}.
By \cite[(5.7)]{S-purity} (with $0$ and $\infty$  interchanged),
we have
\[ \frac{F(\P^1_S,0_S+\infty_S)}{F(\P^1_S,0_S)}
\simeq \frac{F(\P^1_S\setminus 0_S,\infty_S)}{F(\P^1_S\setminus 0_S)}.\]
Hence $a=b+c$, for some  $b\in F(\P^1_S,0_S+\infty_S)$ and $c\in F(\P^1_S\setminus 0_S)$.
It suffices to show $c\in F(\P^1_S)$. 
Let the notations be as in Lemma \ref{lem:diag}. We have 
\[\delta_0:=\beta_0-\gamma_0=\left\{\frac{1-s^{-1}t}{t-1},\theta_1,\ldots,\theta_{d-1}\right\}
\in (V_{d, X_L|0_{S_L}})_{x_{d-1}}.\]
Thus \ref{HS3} yields $(b_L, \delta_0)_{X_L/L, \ux}=0$; furthermore $(c_L, \gamma_0)_{X_L/L, \ux}=0$, by
\eqref{lem:diag2}. Thus the assumption yields
$0=(a_L, \delta_0)_{X_L/L,\ux}= (c_L, \beta_0)_{X_L/L, \ux}$. Hence $c\in F(S)$ by \eqref{lem:diag1}.
\end{proof}

The following proposition is a version of the implication $\ref{thm:LS3}\Rightarrow \ref{thm:LS1}$ in Theorem \ref{thm:LS}
for a modulus pair $(X,D)$ with $X$ smooth and $D$ a {\em reduced} SNCD,
which does not require the existence of a smooth compactification $\Xb$ of $X$ such that $D$ is the restriction
of an SNCD on $\Xb$. Besides Corollary \ref{cor:diag}, an essential ingredient is 
\cite[Corollary 2.5]{SaitologRSC}. Note that though results of the present paper are used in \cite{SaitologRSC},
this is not the case in section 2 of {\em loc. cit.}, hence there is no circular argument.

\begin{prop}\label{prop:diag}
Let $X\in \Sm$ and assume $D$ is a {\em reduced} SNCD on $X$. 
Let $U\subset X$ be an open subset containing all the generic points of $D$.
Let $a\in F(X\setminus D)$. 
\begin{enumerate}[label=(\arabic*)]
\item\label{prop:diag2} 
Assume for all function fields $K/k$ and all $\ux=(x_0,\ldots, x_d)\in\mc(U_K)$ with $x_{d-1}\in D_K^{(0)}$,
we have 
\[(a,\beta)_{X_K/K,\ux}=0,\quad\text{for all }\beta\in K^M_d(\sO_{X_K,x_{d-1}}).\]
Then $a\in F(X)$. 
\item\label{prop:diag1}
Assume for all function fields $K/k$ and all $\ux=(x_0,\ldots, x_d)\in\mc(U_K)$ with $x_{d-1}\in D_K^{(0)}$,
we have 
\[(a,\beta)_{X_K/K, \ux}=0,\quad \text{for all } \beta\in (V_{d, X_K|D_K})_{x_{d-1}}.\]
Then $a\in \tF(X,D)$.
\end{enumerate}
\end{prop}
\begin{proof}
We only prove  \ref{prop:diag1}, the proof of \ref{prop:diag2} is similar.
By \cite[Corollary 2.5]{SaitologRSC} there is an exact sequence
\eq{eq0;lem2;diagonal}{ 0\to \tF(X,D) \to F(X\setminus D) \to \underset{\eta\in D^{(0)}}{\bigoplus} 
\frac{F(\sO_{X,\eta}^h\setminus\eta)}{\tF(\sO_{X,\eta}^h,\eta)}.}
This reduces us to the case $X$ and $D$ smooth, affine, connected, with generic point $\eta\in D$ and it suffices to show that the condition in \ref{prop:diag1} implies $a\in \tF(\sO_{X,\eta}^h,\eta)$.

\begin{claim}\label{prop:diag-claim}
We may further assume that there is a  morphism $X\to D$, such that $D\inj X\to D$ is the identity.
\end{claim}

Indeed, by \cite[Lemma 7.13]{BRS} (which is a variant of \cite[Lemma 8.5]{S-purity} and relies on a result of Elkik \cite{Elkik}),
we find an \'etale morphism $u:X'\to X$ and a morphism $X'\to D$, such that $u$ induces an isomorphism  $u^{-1}(D)\xr{\simeq} D$ and 
the composition  $D\inj X'\to D$ is the identity. 
Thus it suffices to show 
\eq{prop:diag3}{((u^*a)_K, \gamma)_{X'_K/K,\ux}=0\quad \text{for all } \gamma\in (V_{d, X'_K|D_K})_{x_{d-1}},}
for $K/k$ and all $\ux=(x_0,\ldots, x_d)\in \mc(X'_K)$ with $x_{d-1}\in D_K^{(0)}$.
By \ref{HS3} \eqref{prop:diag3} holds for $\gamma\in (V_{d, X'_K|nD_K})_{x_{d-1}}$ for some big enough integer $n$.
Since that natural map
\eq{prop:diag3.5}{(V_{d, X_K|D_K})_{x_{d-1}}\overset{u_K^*}{\longrightarrow} \frac{(V_{d, X'_K|D_K})_{x_{d-1}}}{(V_{d, X'_K|nD_K})_{x_{d-1}}}}
is surjective, it suffices to show \eqref{prop:diag3} for $\gamma=u_K^*\beta$ with $\beta\in (V_{d, X_K|D_K})_{x_{d-1}}$.
Since maximal chains $\ux\in \mc(X'_K)$ with $x_{d-1}\in D_K$ correspond uniquely to maximal chains  in $X_K$ whose 
1-codimensional point is the generic point of $D_K$, the vanishing \eqref{prop:diag3}  follows in this case
directly from the vanishing in \ref{prop:diag1} and the definition of the map $c_{\ux}$ in \eqref{para:hlc2}, which is used in Definition \ref{defn:HLS}. 
This proves claim \ref{prop:diag-claim}.
\medbreak

Shrinking  around the generic point of $D$, we may assume further $D=\Div(t)$ for some $t\in \sO(X)$
defining  an \'etale morphism  $v:X\to \A^1_D$, which satisfies $v^{-1}(0_D)=D$.
By \cite[Lemmas 4.2, 4.4]{S-purity} the morphism 
\[v^*:\frac{F(\A^1_D\setminus 0_D)}{\tF(\A^1_D, 0_D)} \longrightarrow\frac{F(X\setminus D)}{\tF(X,D)}\]
becomes  an isomorphism when we shrink $D$ to its generic point.
Thus we may assume $a=v^*b$ for some $b\in F(\A^1_D\setminus 0_D)$.
By \cite[Lemma 5.9]{S-purity}, we have
\[\frac{F(\A^1_D\setminus 0_D)}{\tF(\A^1_D, 0_D)} \cong \frac{\tF(\P^1_D\setminus 0_D,\infty_D)}{\tF(\P^1_D,0_D+\infty_D)},\]
so we may assume $b\in \tF(\P^1_D\setminus 0_D,\infty_D)$.
It remains to show 
$b\in \tF(\P^1_D, 0_D+\infty_D)$.
By Corollary \ref{cor:diag}\ref{cor:diag2}  it therefore suffices to show,
that for all $K/k$ and all $\ux=(x_0,\ldots, x_{d})\in \mc(\A^1_{D_K})$ with $x_{d-1}\in {D_K}^{(0)}$ (where $D_K$ is embedded in $\A^1_{D_K}$ 
along the zero-section) we have 
\eq{prop:diag4}{(b,\gamma)_{\A^1_{D_K}/K,\ux}=0,\quad \text{for all } \gamma\in (V_{d, \A^1_{D_K}| 0_{D_K}})_{x_{d-1}}.}
To this end we first observe that by a similar argument as was used around \eqref{prop:diag3.5},
the condition in \ref{prop:diag1} also holds with $(V_{d, X_K|D_K})_{x_{d-1}}$ replaced by its Nisnevich stalk
$(V_{d, X_K|D_K})_{x_{d-1}}^h$ and we may as well consider the Nisnevich stalk  of $V_{d, \A^1_{D_K}|0_{D_K}}$ in \eqref{prop:diag4}.
But $v^*$ induces an isomorphism  $\sO^h_{\A^1_{D_K}, 0_{D_K}}\xr{\simeq}\sO^h_{X_K, x_{d-1}}$.
Thus the norm map $K^M_d(K^h_{X_K, x_{d-1}}) \to K^M_d(K^h_{\A^1_{D_K}, 0_{D_K}})$
induces an isomorphism
\[\Nm: (V_{d, X_K|D_K})_{x_{d-1}}^h\stackrel{\simeq}{\longrightarrow} (V_{d, \A^1_{D_K}|0_{D_K}})^h_{x_{d-1}}.\]
Now let $\gamma$ and $\ux\in \mc(\A^1_{D_K})$ be as in  \eqref{prop:diag4}.
By the construction we can lift  $\ux$ uniquely to $\uy\in \mc(X_K)$.
Take $\beta\in (V_{d, X_K|D_K})_{x_{d-1}}^h$ with $\Nm(\beta)=\gamma$.
Then
\[0=(a, \beta)_{X_K/K,\uy}= (v^*b, \beta)_{X_K/K,\uy}= (b, \Nm(\beta))_{\A^1_{D_K}/K,\ux}= (b, \gamma)_{\A^1_{D_K}/K,\ux}\]
where  the first equality holds by the condition in \ref{prop:diag1} and the third equality holds by \ref{HS5'} in Corollary \ref{cor:HS5}.
This yields \eqref{prop:diag4} and completes the proof.
\end{proof}


\providecommand{\bysame}{\leavevmode\hbox to3em{\hrulefill}\thinspace}
\providecommand{\MR}{\relax\ifhmode\unskip\space\fi MR }
\providecommand{\MRhref}[2]{%
  \href{http://www.ams.org/mathscinet-getitem?mr=#1}{#2}
}
\providecommand{\href}[2]{#2}

\end{document}